\documentclass{amsart}
\usepackage{graphicx}
\usepackage{amssymb}
\usepackage{wrapfig}
\usepackage[bookmarksnumbered,colorlinks,plainpages,backref]{hyperref}

\newtheorem{thm}{Theorem}[section]
\newtheorem{cor}[thm]{Corollary}
\newtheorem{lem}[thm]{Lemma}
\newtheorem{prop}[thm]{Proposition}
\theoremstyle{definition}
\newtheorem{defn}[thm]{Definition}
\theoremstyle{remark}
\newtheorem{rem}[thm]{Remark}
\newtheorem{question}[thm]{Question}
\numberwithin{equation}{section}
\newcommand{\norm}[1]{\left\Vert#1\right\Vert}
\newcommand{\abs}[1]{\left\vert#1\right\vert}
\newcommand{\set}[1]{\left\{#1\right\}}

\newcommand{\dbar}{\bar\partial}
\newcommand{\ddbar}{\partial\bar\partial}

\newcommand{\sff}{\mathrm{I\!I}}

\DeclareMathOperator{\range}{Range}
\DeclareMathOperator{\re}{Re}
\DeclareMathOperator{\im}{Im}
\DeclareMathOperator{\supp}{supp}
\DeclareMathOperator{\dist}{dist}

\DeclareMathOperator{\Span}{span}

\DeclareMathOperator{\Hess}{Hess}
\begin{document}

\title[The Diederich-Forn{\ae}ss Index on $C^2$ Domains in Hermitian Manifolds]{The Strong Diederich-Forn{\ae}ss Index on $C^2$ Domains in Hermitian Manifolds}%
\author{Phillip S. Harrington}%
\address{SCEN 309, 1 University of Arkansas, Fayetteville, AR 72701}%
\email{psharrin@uark.edu}%


\begin{abstract}
  For a relatively compact Stein domain $\Omega$ with $C^2$ boundary in a Hermitian manifold $M$, we consider the strong Diederich-Forn{\ae}ss index, denoted $DF(\Omega)$: the supremum of all exponents $0<\eta<1$ such that eigenvalues of the complex Hessian of $-(-\rho)^\eta$ are bounded below by some positive multiple of $(-\rho)^\eta$ on $\Omega$ for some $C^2$ defining function $\rho$.  We will show that $DF(\Omega)$ is completely characterized by the existence of a Hermitian metric with curvature terms satisfying a certain inequality when restricted to the null-space of the Levi-form.
\end{abstract}
\maketitle

\tableofcontents


\section{Introduction}

Let $M$ be a Hermitian manifold of dimension $n\geq 2$ and let $\Omega\subset M$ be a relatively compact domain with $C^2$ boundary satisfying the strong Oka property: there exists a $C^2$ defining function $\rho$ for $\Omega$ such that
\begin{equation}
\label{eq:Strong_Oka_Property}
  i\ddbar(-\log(-\rho))\geq C\omega
\end{equation}
on $\Omega$ for some constant $C>0$.  We say that $0<\eta<1$ is a strong Diederich-Forn{\ae}ss exponent for $\Omega$ if there exists a $C^2$ defining function $\rho$ for $\Omega$ and a constant $C>0$ such that
\begin{equation}
\label{eq:DF_Index}
  i\ddbar(-(-\rho)^\eta)\geq \eta C(-\rho)^\eta\omega
\end{equation}
on $\Omega$, where $\omega$ denotes the K\"ahler form for the Hermitian metric on $M$.  The strong Diederich-Forn{\ae}ss index for $\Omega$, denoted $DF(\Omega)$, is the supremum over all strong Diederich-Forn{\ae}ss exponents for $\Omega$.  We note that this index is strong in the sense of both \cite{Har22} (we require $\rho$ to be at least $C^2$) and \cite{AdYu21} (we require $-(-\rho)^\eta$ to be strictly plurisubharmonic).  We include the factor of $\eta$ in the right-hand side of \eqref{eq:DF_Index} so that \eqref{eq:DF_Index} implies that $\rho$ satisfies \eqref{eq:Strong_Oka_Property} on $\Omega$ for the same value of $C$.  The relationship between the strong Oka property and the strong Diederich-Forn{\ae}ss index has been studied in several sources: see \cite{HaSh07} or \cite{FuSh18} and the references therein.

The strong Diederich-Forn{\ae}ss index first appeared in \cite{DiFo77b}, in which Diederich and Forn{\ae}ss proved that $DF(\Omega)>0$ for any bounded pseudoconvex domain in a Stein manifold with $C^2$ boundary.  In \cite{DiFo77a}, Diederich and Forn{\ae}ss proved that for every $0<\eta<1$ there exists a bounded domain $\Omega$ in $\mathbb{C}^n$ with smooth boundary such that $DF(\Omega)<\eta$; these are the famous worm domains (see Section \ref{sec:Examples} below).  Ohsawa and Sibony showed that $DF(\Omega)>0$ for any domain in $\mathbb{CP}^n$ with $C^2$ boundary in \cite{OhSi98}.  Their starting point is Takeuchi's Theorem \cite{Tak64}, which states that the signed distance function satisfies the strong Oka property \eqref{eq:Strong_Oka_Property} for any pseudoconvex domain in $\mathbb{C}\mathbb{P}^n$.  Takeuchi's Theorem has been generalized in \cite{Ele75}, \cite{Suz76}, and \cite{GrWu78}, so it is now known to hold on any manifold with positive holomorphic bisectional curvature.  An alternate proof of Takeuchi's Theorem for domains with $C^2$ boundary was given in \cite{CaSh05}.

Much recent interest in the strong Diederich-Forn{\ae}ss index stems from a key result of Berndtsson and Charpentier \cite{BeCh00}, in which they prove that the Bergman Projection is continuous in the Sobolev space $W^s(\Omega)$ for any $0\leq s<\frac{1}{2}DF(\Omega)$ when $\Omega\subset\mathbb{C}^n$ is a bounded domain with $C^2$ boundary.  In fact, their result only requires a Lipschitz boundary, but we will not consider this possibility in the present paper, except to note that this motivates key results in \cite{Har08a}, \cite{Har17}, and \cite{Har22}.  In \cite{CSW04}, Cao, Shaw, and Wang extended the result of Berndtsson and Charpentier to domains in $\mathbb{C}\mathbb{P}^n$.  When $\partial\Omega$ is smooth, there are many results relating the Diederich-Forn{\ae}ss index to higher order Sobolev regularity.  See, for example, \cite{Koh99}, \cite{Har11}, \cite{PiZa14}, \cite{Liu22}, and \cite{LiSt22}.  Liu and Straube have shown in \cite{LiSt22} that when $\partial\Omega$ is smooth, the Bergman Projection on $(0,n-2)$-forms $P_{n-2}$ is continuous in $W^s_{(0,n-2)}(\Omega)$ for all $s>0$ if $DF(\Omega)=1$.  In particular, this means that when $n=2$ the classical Bergman Projection is continuous in $W^s(\Omega)$ for all $s>0$ if $DF(\Omega)=1$.  Current results in higher dimensions require additional hypotheses, but it is known that when $\Omega\subset\mathbb{C}^n$ admits a defining function which is plurisubharmonic on $\partial\Omega$, we have both $DF(\Omega)=1$ (see \cite{HeFo07} and \cite{HeFo08}) and regularity for the Bergman Projection in $W^s(\Omega)$ for all $s\geq 0$ (see \cite{BoSt91}), so there is great interest in understanding the relationship between the condition $DF(\Omega)=1$ and global regularity for the Bergman Projection.

Classically, we say that $0<\eta<1$ is a Diederich-Forn{\ae}ss exponent for $\Omega$ if there exists a $C^2$ defining function $\rho$ for $\Omega$ and a constant $C>0$ such that $-(-\rho)^\eta$ is strictly plurisubharmonic on $\Omega$ (as in \cite{DiFo77b}).  We note that in either a Stein or projective manifold, if \eqref{eq:DF_Index} holds with $C=0$ then it holds for some $C>0$ (for a different $C^2$ defining function $\rho$).  Hence the classical definition is equivalent to the strong definition that we have given in these contexts.  As shown by Adachi and Yum \cite{AdYu21}, great care must be taken to distinguish these definitions when working in more general classes of manifolds.  In particular, the existence of a $C^2$ defining function $\tilde\rho$ such that $-(-\tilde\rho)^\eta$ is strictly plurisubharmonic on $\Omega$ ($DF_s(\Omega)>\eta$ in the terminology of \cite{AdYu21}) does not necessarily imply the existence of a $C^2$ defining function $\rho$ satisfying \eqref{eq:DF_Index} on general Hermitian manifolds ($DF_s(\partial\Omega)>\eta$ in the terminology of \cite{AdYu21}).  However, as shown in Proposition 3.2 in \cite{AdYu21}, the strong Oka property \eqref{eq:Strong_Oka_Property} implies that the hypotheses of Theorem 2 in \cite{AdYu21} are satisfied, and hence the various definitions of the Diederich-Fornaess index considered by Adachi and Yum agree in this setting.  Since our applications to the Bergman Projection require \eqref{eq:DF_Index} for $C>0$, we restrict to settings in which the strong Oka property holds throughout the present paper.  We note that several sources, including \cite{AdYu21}, only require that strict plurisubharmonicity of $-(-\rho)^\eta$ hold in an interior neighborhood of $\partial\Omega$.  We will see in Propositions \ref{prop:DF_Index_necessary_condition_boundary} and \ref{prop:DF_Index_sufficient_condition_boundary} that our key results also hold for this \textit{a priori} weaker definition.  However, our applications of the Diederich-Forn{\ae}ss index only apply on a Stein domain, and we will see in the proof of Lemma \ref{lem:DF_Index_sufficient_condition} that whenever \eqref{eq:DF_Index} holds on an interior neighborhood of the boundary of a Stein domain, we can obtain a defining function such that \eqref{eq:DF_Index} holds on the entire domain.

In order to facilitate computations of the Diederich-Forn{\ae}ss index, Liu showed that on domains with $C^3$ boundary it is equivalent to consider the existence of a function satisfying a differential inequality on $\partial\Omega$ (\cite{Liu19a} and \cite{Liu19b}).  Using this reduction, Liu was able to explicitly compute the Diederich-Forn{\ae}ss index on the worm domain in \cite{Liu19a}.  In \cite{Yum21}, Yum reformulated Liu's result in terms of D'Angelo's one-form $\alpha$ \cite{DAn79,DAn87} (see Section 5.9 in \cite{Str10} for further background).  Boas and Straube had already shown that this one-form plays a role in the regularity theory for the Bergman Projection in \cite{BoSt93}.  One of our two goals in the present paper is to prove that the result of Liu as reformulated by Yum still holds for any bounded domain with $C^2$ boundary.  Our second goal is to reframe the result of Liu as reformulated by Yum in terms of the geometry of $\partial\Omega$ and $M$.  Although the Diederich-Forn{\ae}ss index itself is independent of the choice of metric, the existence of metrics with certain properties can reveal invariant properties of the ambient manifold, as has been shown in the study of pseudoconvex domains in manifolds with positive holomorphic bisectional curvature.

When studying the Diederich-Forn{\ae}ss index, it is essential to understand how the Levi-form changes near the boundary of $\Omega$.  This is the primary obstacle to working with $C^2$ boundaries: computing the rate of change of the Levi-form requires three derivatives of the defining function.  This is the essence of Range's observation that a simpler, more flexible proof of Diederich and Forn{\ae}ss's original result can be obtained when the boundary is $C^3$ \cite{Ran81}.  The result of Liu \cite{Liu19a} as reformulated by Yum \cite{Yum21} relies on derivatives of the Levi-form, so it also requires a $C^3$ boundary.  In contrast, results on domains with $C^2$ boundaries tend to rely on special properties of the signed distance function, either because in Stein or projective manifolds $-\log(-r)$ is plurisubharmonic on $\Omega$ when $r$ is the signed distance function (as in \cite{DiFo77b}), or via the fact that the eigenvalues of the Hessian of $r$ satisfy a Riccati equation (see for example \cite{Wei75}, \cite{HeMc12}, or \cite{Har19}).  As shown in \cite{CaSh05}, these two properties of the signed distance function are closely linked.

In the present paper, we wish to recover some of the flexibility of Range's method from \cite{Ran81} when working with $C^2$ boundaries.  Given a domain $\Omega$ with $C^2$ boundary, we say that a $C^2$ defining function $r$ for $\Omega$ on some neighborhood $U$ of $\partial\Omega$ is \textit{admissible} if $|dr|$ is $C^2$ on $U$ with respect to the given Hermitian metric.  The signed distance function is always admissible because $|dr|\equiv 1$ on a neighborhood of $\partial\Omega$ (although a result of Krantz and Parks \cite{KrPa81} is essential here).  This gives us a large class of admissible defining functions.

Given an admissible defining function $r$ for $\Omega$ on a neighborhood $U$ of $\partial\Omega$, let $L_r$ be the unique $C^1$ section of $T^{1,0}(U)$ satisfying
\begin{equation}
\label{eq:metric_compatibility_complex}
  |\partial r|^{-2}\partial r(Z)=\left<Z,L_r\right>\text{ on }U\text{ for all }Z\in T^{1,0}(U).
\end{equation}
We define a real-valued $1$-form $\alpha_r\in\Lambda^1(U)$ by
\begin{equation}
\label{eq:alpha_r_defn}
  \alpha_r(Z)=\ddbar r(Z,\bar L_r)\text{ and }\alpha_r(\bar Z)=\ddbar r(L_r,\bar Z)\text{ on }U\text{ for all }Z\in T^{1,0}(U).
\end{equation}
On $\partial\Omega$, this is equivalent to the usual definition by (5.85) in \cite{Str10}.  We will show in Proposition \ref{prop:alpha_closed} that the restriction of $\alpha_r$ to the tangent space of any complex submanifold in the boundary of $\Omega$ is closed in the weak sense.  This follows from Lemma 5.14 in \cite{Str10} when the boundary is at least $C^3$, but is a new result on $C^2$ boundaries.  Even though $\alpha_r$ is only continuous, we will show in Proposition \ref{prop:beta} that there exists a continuous, real-valued $2$-form $\beta_r\in\Lambda^2(U)$ with the property that
\begin{equation}
\label{eq:beta_C3_defn}
  \beta_r=-\frac{i}{2}(\partial\alpha_r-\dbar\alpha_r)
\end{equation}
on $U$ in the weak sense.  We will see (see Lemma \ref{lem:geometric_invariance} and the preceding discussion) that $\alpha_r$ and $\beta_r$ are independent of the choice of metric when restricted to the null-space of the Levi-form.  When restricted to a complex submanifold contained in $\partial\Omega$, the de Rham cohomology class represented by $\alpha_r$ is independent of the choice of $r$ (see \eqref{eq:alpha_r_geometric} below), and the Bott-Chern cohomology class represented by $\beta_r$ is independent of the choice of $r$ (see \eqref{eq:beta_r_geometric} below).

For any $P\in\partial\Omega$, the null-space of the Levi-form $\mathcal{N}_P(\partial\Omega)$ is defined to be the set of all $Z\in T_P^{1,0}(\partial\Omega)$ such that $\ddbar\rho(Z,\bar W)=0$ for all $W\in T_P^{1,0}(\partial\Omega)$, where $\rho$ is any $C^2$ defining function for $\Omega$.

With these definitions, we are ready to state a $C^2$ version of the result of Liu \cite{Liu19a} as reformulated by Yum \cite{Yum21}:
\begin{thm}
\label{thm:boundary_equivalence}
  Let $M$ be a Hermitian manifold of complex dimension $n\geq 2$ and let $\Omega\subset M$ be a relatively compact Stein domain with $C^2$ boundary.  Let $U$ be a neighborhood of $\partial\Omega$ and let $r\in C^2(U)$ be an admissible defining function for $\Omega$.  For $0\leq\eta<1$, the following are equivalent:
  \begin{enumerate}
    \item $\eta<DF(\Omega)$,
    \item there exists $C>0$ and $h\in C^\infty(U)$ such that
      \begin{equation}
      \label{eq:DF_Index_boundary}
        -i\beta_r(Z,\bar Z)+\ddbar h(Z,\bar Z)\geq\frac{\eta}{1-\eta}\abs{\partial h(Z)-\alpha_r(Z)}^2+C|Z|^2
      \end{equation}
      on $\partial\Omega$ for all $Z\in T^{1,0}(\partial\Omega)$, and
    \item there exists $C>0$ and $h\in C^2(\partial\Omega)$ such that \eqref{eq:DF_Index_boundary} holds for all $P\in\partial\Omega$ and $Z\in\mathcal{N}_P(\partial\Omega)$.
  \end{enumerate}
  In case (2) or (3), either $-(-re^{-h})^\eta$ (when $\eta>0$) or $-\log(-re^{-h})$ (when $\eta=0$) is strictly plurisubharmonic on $\tilde U\cap\Omega$ for some neighborhood $\tilde U\subset U$ of $\partial\Omega$.
\end{thm}
As shown in \cite{Liu19b} and \cite{Har22}, Theorem \ref{thm:boundary_equivalence} leads to several natural sufficient conditions for $DF(\Omega)=1$.  For example, if $\mathcal{N}_P(\partial\Omega)$ is contained in the tangent space to a complex submanifold of $\partial\Omega$ for all $P\in\partial\Omega$ and $\alpha_r$ is exact on this complex submanifold, then we can choose $h$ so that $dh(Z)=\alpha_r(Z)$ for all $Z\in\mathcal{N}_P(\partial\Omega)$ and $P\in\partial\Omega$.  This is closely related to regularity properties of the Bergman Projection \cite{BoSt93} and the existence of defining functions that are plurisubharmonic on the boundary (\cite{Mer20} and \cite{GaHa22}).  On the other hand, if $\partial\Omega$ satisfies McNeal's Property $(\tilde P)$ \cite{McN02}, then for every $C>0$ and $0<\eta<1$ there exists a function $h\in C^\infty(U)$ such that \eqref{eq:DF_Index_boundary} holds on $\partial\Omega$ for all $Z\in T^{1,0}(U)$.  This is also strong enough to imply the existence of a Stein neighborhood base for $\overline\Omega$ by adapting an argument of Sibony \cite{Sib87}, so in this special case we can take $M$ to be a Stein manifold.  If $\partial\Omega$ is $C^m$ for some $m\geq 3$ and $r$ is a $C^m$ defining function, then $r$ is trivially admissible, so for any $0<\eta<DF(\Omega)$ we may choose $h\in C^\infty(U)$ satisfying \eqref{eq:DF_Index_boundary} on $\partial\Omega$ and obtain a $C^m$ defining function $\rho=re^{-h}$ satisfying \eqref{eq:DF_Index}.  This generalizes the proof of the second statement in Theorem 1.3 in \cite{Har22} to Hermitian manifolds.

As a consequence of Theorem \ref{thm:boundary_equivalence}, we are able to generalize a result of Adachi and Yum (Proposition 5.7 in \cite{AdYu21}) to domains with $C^2$ boundaries:
\begin{cor}
\label{cor:complex_submanifold_Kahler}
  Let $M$ be a Hermitian manifold of complex dimension $n\geq 2$ and let $\Omega\subset M$ be a relatively compact Stein domain with $C^2$ boundary.  If $DF(\Omega)>0$, then $\partial\Omega$ contains no compact complex submanifolds of positive dimension.
\end{cor}
When we describe a complex submanifold as compact, we always mean compact relative to the intrinsic topology on the submanifold, so Corollary \ref{cor:complex_submanifold_Kahler} does not exclude the possibility that $\partial\Omega$ might contain a lamination by a non-compact complex submanifold.  As noted above, Ohsawa and Sibony \cite{OhSi98} have shown that $DF(\Omega)>0$ for any pseudoconvex domain $\Omega\subset\mathbb{C}\mathbb{P}^n$ with $C^2$ boundary, so Corollary \ref{cor:complex_submanifold_Kahler} guarantees that no pseudoconvex domain in $\mathbb{C}\mathbb{P}^n$ with $C^2$ boundary can contain a compact complex submanifold of positive dimension in the boundary.  This is a special case of Theorem 1.3 in \cite{DiOh02}.

On the other hand, Diederich and Ohsawa \cite{DiOh82} have constructed a family of relatively compact Stein domains with real-analytic boundaries that contain compact complex submanifolds in their boundaries (see Section \ref{sec:Examples} below for a special case).  For such a domain $\Omega$, Corollary \ref{cor:complex_submanifold_Kahler} implies that $DF(\Omega)=0$.  The construction of Diederich and Ohsawa is closely related to the construction of the worm domain by Diederich and Forn{\ae}ss in \cite{DiFo77a}, so it would be interesting to know the following (motivated by the result of Barrett in \cite{Bar92}):
\begin{question}
  If $M$ is a Hermitian manifold and $\Omega\subset M$ is a relatively compact Stein domain with $C^2$ boundary such that $DF(\Omega)=0$, does it follow that the Bergman Projection is discontinuous in $W^s(\Omega)$ for all $s>0$?
\end{question}

To state our next main result, we need some additional notation.  Denote the complex structure on $M$ by $J$ and the Hermitian metric on $T(M)$ by  $\left<\cdot,\cdot\right>$.  Let $\omega$ denote the K\"ahler form for the Hermitian metric on $M$, i.e., $\omega(Z,\bar W)=i\left<Z,W\right>$ for all $Z,W\in T^{1,0}(M)$.  Let $\nabla$ denote the Chern connection on $T^{1,0}(M)$ with respect to this Hermitian metric, as well as the unique affine connection on $T(M)$ induced by the Chern connection.  For $X,Y\in T(M)$ we define the torsion tensor
\[
  T_\nabla(X,Y)=\nabla_X Y-\nabla_Y X-[X,Y]
\]
and the curvature tensor
\[
  R_\nabla(X,Y)=\nabla_X\nabla_Y-\nabla_Y\nabla_X-\nabla_{[X,Y]}.
\]
Let $\Omega\subset M$ be a relatively compact Stein domain with $C^2$ boundary.  Let $\nu_{\mathbb{R}}$ denote a $C^1$ section of $T(M)$ such that $\nu_{\mathbb{R}}|_{\partial\Omega}$ is the unit outward normal with respect to the given metric on $M$.  For $X,Y\in T(\partial\Omega)$, the Second Fundamental Form for $\partial\Omega$ is defined by
\[
  \sff_\nabla(X,Y)=-\left<\nabla_X\nu_{\mathbb{R}},Y\right>\nu_{\mathbb{R}}.
\]
We extend each of these tensors and forms to $\mathbb{C}T(M)$ via $\mathbb{C}$-linearity.  Throughout this paper, all tensors are taken to be $\mathbb{C}$-linear in their arguments, with the notable exception of the sesquilinear Hermitian inner product.  Define the complex normal $\nu_{\mathbb{C}}\in T^{1,0}(M)$ by $\nu_{\mathbb{C}}=\frac{1}{\sqrt{2}}(\nu_{\mathbb{R}}-iJ\nu_{\mathbb{R}})$, so that $|\nu_{\mathbb{C}}|=1$ on $\partial\Omega$ and $\left<\nu_{\mathbb{C}},Z\right>=0$ on $\partial\Omega$ whenever $Z\in T^{1,0}(\partial\Omega)$.

Our second main result is the following:
\begin{thm}
\label{thm:geometric_equivalence}
  Let $M$ be a complex manifold of complex dimension $n\geq 2$ and let $\Omega\subset M$ be a relatively compact Stein domain with $C^2$ boundary.  For $0\leq\eta<1$, the following are equivalent:
  \begin{enumerate}
    \item $\eta<DF(\Omega)$,
    \item there exists a Hermitian metric on $T(M)$ and a constant $C>0$ such that
        \begin{equation}
        \label{eq:DF_Index_boundary_geometry}
          \sum_{j=1}^{n-1}\abs{\sff_\nabla(Z,W_j)}^2+\frac{1}{2}\left<R_\nabla(Z,\bar Z)\nu_{\mathbb{C}},\nu_{\mathbb{C}}\right>\geq
          \frac{\eta}{1-\eta}\abs{\sff_\nabla(Z,J\nu_\mathbb{R})}^2+C|Z|^2
        \end{equation}
        for all $P\in\partial\Omega$ and $Z\in\mathcal{N}_P(\partial\Omega)$, where $\set{W_j}_{j=1}^{n-1}$ is an orthonormal basis for $T^{1,0}_P(\partial\Omega)$,
    \item there exists a Hermitian metric on $T(M)$ and a constant $C>0$ such that
        \begin{equation}
        \label{eq:DF_Index_boundary_vector_field}
          \frac{1}{2}\abs{\nabla_{\bar Z}\nu_{\mathbb{C}}-\left<\nabla_{\bar Z}\nu_{\mathbb{C}},\nu_{\mathbb{C}}\right>\nu_{\mathbb{C}}}^2+\frac{1}{2}\left<R_\nabla(Z,\bar Z)\nu_{\mathbb{C}},\nu_{\mathbb{C}}\right>\geq
          \frac{\eta}{1-\eta}\abs{\left<\nabla_{\bar Z}\nu_{\mathbb{C}},\nu_{\mathbb{C}}\right>}^2+C|Z|^2
        \end{equation}
        for all $P\in\partial\Omega$ and $Z\in\mathcal{N}_P(\partial\Omega)$, and
    \item there exists a Hermitian metric on $T(M)$ and a constant $C>0$ such that the signed distance function $\rho(z)=-\dist(z,\partial\Omega)$ for $z\in\overline\Omega$ and $\rho(z)=\dist(z,\partial\Omega)$ for $z\notin\overline\Omega$ satisfies \eqref{eq:DF_Index} (when $\eta>0$) or \eqref{eq:Strong_Oka_Property} (when $\eta=0$) on $U\cap\Omega$ for some neighborhood $U$ of $\partial\Omega$.
  \end{enumerate}
\end{thm}
We will see that \eqref{eq:DF_Index_boundary_geometry} and \eqref{eq:DF_Index_boundary_vector_field} are the same estimate with different notation: \eqref{eq:DF_Index_boundary_geometry} is written in terms of the extrinsic curvature of $\partial\Omega$ and \eqref{eq:DF_Index_boundary_vector_field} is written in terms of the components of $\dbar\nu_{\mathbb{C}}$.  We note that the crucial term $\left<\nabla_{\bar Z}\nu_{\mathbb{C}},\nu_{\mathbb{C}}\right>$ appearing in \eqref{eq:DF_Index_boundary_vector_field} is closely related to the good vector field condition introduced by Boas and Straube in \cite{BoSt91} (see \cite{BoSt99} and \cite{Str10} for generalizations of this condition).  As explained in \cite{Har19}, if the set of weakly pseudoconvex points is suitably well-behaved, this observation can be used to show that $DF(\Omega)=1$ when $\Omega\subset\subset\mathbb{C}^n$ has a smooth boundary and admits a family of good vector fields in the sense of \cite{BoSt99}.

Ohsawa and Sibony have already shown that the Diederich-Forn{\ae}ss index is positive on domains with $C^2$ boundary in manifolds with strictly positive holomorphic bisectional curvature (see Theorem 1.1 in \cite{OhSi98}).  The holomorphic bisectional curvature of $M$ is strictly positive precisely when $\frac{\left<R_\nabla(Z,\bar Z)W,W\right>}{|Z|^2|W|^2-|\left<Z,W\right>|^2}>0$ for all linearly independent $Z,W\in T^{1,0}(M)$, and hence \eqref{eq:DF_Index_boundary_geometry} holds for some $\eta>0$ sufficiently small whenever the holomorphic bisectional curvature of $M$ is strictly positive.  Note that every projective manifold admits a metric for which the holomorphic bisectional curvature is positive by pulling back the Fubini-Study metric via the inclusion map.  Furthermore, every relatively compact domain in a Stein manifold can be embedded in a (possibly incomplete) projective manifold by embedding the Stein manifold in $\mathbb{C}^n$ and treating $\mathbb{C}^n$ as a local coordinate chart for a neighborhood in $\mathbb{C}\mathbb{P}^n$, so Theorem \ref{thm:geometric_equivalence} implies both the original result of Diederich and Forn{\ae}ss \cite{DiFo77b} and the analogous result proven by Ohsawa and Sibony \cite{OhSi98}.
\begin{question}
  Does there exist a Hermitian manifold $M$ and a relatively compact Stein domain $\Omega\subset M$ with $C^2$ boundary such that $DF(\Omega)>0$ and there does not exist a K\"ahler metric on any neighborhood of $\overline\Omega$?  Does there exist such a domain such that no neighborhood of $\overline\Omega$ admits a metric with positive holomorphic bisectional curvature?
\end{question}

In Section \ref{sec:Examples}, we will look at several examples that illustrate the main ideas behind this paper.  In Section \ref{sec:Third_derivatives}, we will generalize the Hessian operator to a third-order differential operator and study the symmetries of this operator.  In Section \ref{sec:regularity_properties}, we will look carefully at those special cases in which a third derivative of a $C^2$ defining function will exist.  Sections \ref{sec:alpha} and \ref{sec:beta} will establish the basic properties of the forms $\alpha_r$ and $\beta_r$.  Once this groundwork has been established, Section \ref{sec:level_curves} can finally develop the relationship between the Levi-form of a level curve of an admissible defining function and the Levi-form on the boundary.  This allows us to derive necessary and sufficient conditions for $0<\eta<1$ to be a Diederich-Forn{\ae}ss exponent (or for the strong Oka property to hold) in Sections \ref{sec:DF_index_necessary} and \ref{sec:DF_index_sufficient}.  The proofs of the theorems and corollaries discussed in this introduction are all presented in Section \ref{sec:proofs}.  We conclude with an appendix that outlines the proofs of some well-known results in order to verify that our low regularity hypotheses will suffice.

\section{Examples}
\label{sec:Examples}

Before proving our main results, we will consider several illustrative examples.  First we consider the well-known worm domain of Diederich and Forn{\ae}ss \cite{DiFo77a}.  See 6.4 in \cite{ChSh01} for a detailed exposition of the worm domain and its properties.  We note that the worm domain is traditionally parameterized by a parameter denoted $\beta$, but since $\beta$ has a different meaning in this paper we will replace this parameter with $\gamma$.

Fix $\gamma>\frac{\pi}{2}$.  Let $\lambda\in C^\infty(\mathbb{R})$ be an even, convex function such that $\lambda(x)=0$ when $|x|\leq\gamma-\frac{\pi}{2}$ and $\lambda(x)>0$ when $|x|>\gamma-\frac{\pi}{2}$.  For $z\in\mathbb{C}^2$ such that $z_2\neq 0$, we set
\[
  r_\gamma(z)=\abs{z_1+e^{i\log|z_2|^2}}^2-1+\lambda(\log|z_2|^2),
\]
and we define
\[
  \Omega_\gamma=\set{(z_1,z_2)\in\mathbb{C}^2:z_2\neq 0\text{ and }r_\gamma(z)<0}.
\]
The domain $\Omega_\gamma$ is a bounded pseudoconvex domain with smooth boundary, and the boundary contains the complex submanifold
\[
  S_\gamma=\set{(z_1,z_2)\in\mathbb{C}^2:z_1=0\text{ and }\abs{\log|z_2|^2}<\gamma-\frac{\pi}{2}}.
\]
Moreover, $\partial\Omega_\gamma\backslash\bar S_\gamma$ is strictly pseudoconvex, so it suffices to analyze $\partial\Omega_\gamma$ on the set $S_\gamma$.  Using a version of Theorem \ref{thm:boundary_equivalence}, Liu has explicitly computed the Diederich-Forn{\ae}ss Index of $\Omega_\gamma$ to be $DF(\Omega_\gamma)=\frac{\pi}{2\gamma}$.  We will construct a K\"ahler metric in a neighborhood of $\overline{\Omega_\gamma}$ satisfying \eqref{eq:DF_Index_boundary_geometry}.

Fix $0<\eta<\frac{\pi}{2\gamma}$.  Then there exists a parameter $\frac{1}{\eta}-1>t>\frac{2\gamma}{\pi}-1$.  Let $f\in C^\infty(\mathbb{R})$ be an even, positive-valued function such that $f(x)=(\cos(t^{-1}x))^{2t}$ on some neighborhood of the set $\left[\frac{\pi}{2}-\gamma,\gamma-\frac{\pi}{2}\right]$.  When $f(x)=(\cos(t^{-1}x))^{2t}$, we compute $f'(x)=-2f(x)\tan(t^{-1}x)$.  For some constant $s>0$, we define our K\"ahler metric via the K\"ahler form
\[
  \omega_\gamma=i\ddbar\left(|z_1|^2 f(\log|z_2|^2)+s\log|z_2|^2\right),
\]
so that $d\omega_\gamma\equiv 0$ and we may expand to obtain
\begin{multline*}
  \omega_\gamma=i f(\log|z_2|^2)dz_1\wedge d\bar z_1+i\frac{z_1}{z_2}f'(\log|z_2|^2)dz_2\wedge d\bar z_1+i\frac{\bar z_1}{\bar z_2} f'(\log|z_2|^2)dz_1\wedge d\bar z_2\\
  +i\left(\frac{|z_1|^2}{|z_2|^2} f''(\log|z_2|^2)+\frac{s}{|z_2|^2}\right)dz_2\wedge d\bar z_2.
\end{multline*}
We may choose $s>0$ sufficiently large so that this form is positive definite on a neighborhood of $\overline{\Omega_\gamma}$, and denote this neighborhood $M_\gamma$.  As the notation suggests, we will view $M_\gamma$ as an (incomplete) K\"ahler manifold containing $\overline{\Omega_\gamma}$.  Observe that
\begin{equation}
\label{eq:omega_gamma_on_S}
  \omega_\gamma|_{S_\gamma}=if(\log|z_2|^2)dz_1\wedge d\bar z_1+i \frac{s}{|z_2|^2} dz_2\wedge d\bar z_2.
\end{equation}

We compute
\[
  \partial r_\gamma(z)|_{S_\gamma}=e^{-i\log|z_2|^2}dz_1.
\]
If we set $Z=\frac{\partial}{\partial z_2}$, then $Z|_{S_\gamma}$ spans $T^{1,0}(S_\gamma)$.  Hence, any formula which holds on $S_\gamma$ may be differentiated by $Z$.  The unique vector field $L_{r_\gamma}$ satisfying \eqref{eq:metric_compatibility_complex} must satisfy $L_{r_\gamma}|_{S_\gamma}=e^{i\log|z_2|^2}\frac{\partial}{\partial z_1}$.  We also have
\[
  \ddbar r_\gamma(z)|_{S_\gamma}=dz_1\wedge d\bar z_1-\frac{i}{\bar z_2}e^{-i\log|z_2|^2}dz_1\wedge d\bar z_2+\frac{i}{z_2}e^{i\log|z_2|^2}dz_2\wedge d\bar z_1,
\]
so $\alpha_{r_\gamma}(Z)|_{S_\gamma}=\frac{i}{z_2}$.  If $\iota_{S_\gamma}:S_\gamma\rightarrow M_\gamma$ is the inclusion map, then we have $\iota_{S_\gamma}^*\alpha_{r_\gamma}=\frac{i}{z_2}dz_2-\frac{i}{\bar z_2}d\bar z_2$, and hence $\iota_{S_\gamma}^*\beta_{r_\gamma}=0$ (see \eqref{eq:beta_C3_defn} for the definition).  Clearly $\iota_{S_\gamma}^*\alpha_{r_\gamma}$ is $d_{S_\gamma}$-closed (see Proposition \ref{prop:alpha_closed} below), but this is not $d_{S_\gamma}$-exact, even in the weak sense, since there does not exist a continuous branch of $-2\im\log z$ on $S_\gamma$.  Hence $[\iota_{S_\gamma}^*\alpha_{r_\gamma}]$ represents a non-trivial cohomology class in $H^1_{dR}(S_\gamma)$.  As observed in \cite{BoSt93}, this lies at the heart of the many pathologies of the worm domain.  We will see in Corollary \ref{cor:beta_on_submanifold} that $\iota_{S_\gamma}^*\beta_{r_\gamma}$ always represents a trivial cohomology class in $H^2_{dR}(S_\gamma)$, so it is of greater interest to check if $\iota_{S_\gamma}^*\beta_{r_\gamma}$ is $\partial_{S_\gamma}\dbar_{S_\gamma}$-exact, i.e., if $[\iota_{S_\gamma}^*\beta_{r_\gamma}]$ represents a non-trivial element of the Bott-Chern cohomology.  On the worm domain, $\iota_{S_\gamma}^*\beta_{r_\gamma}$ is trivial, so the Bott-Chern cohomology class represented by $\iota_{S_\gamma}^*\beta_{r_\gamma}$ is also trivial.  This tells us that, in general, we should not expect $\beta_r$ to contain the same cohomological information as $\alpha_r$, a fact which we will exploit in the proof of Corollary \ref{cor:complex_submanifold_Kahler}.

For the Chern connection, we necessarily have
\[
  \nabla_{\bar Z}L_{r_\gamma}|_{S_\gamma}=\bar Z(e^{i\log|z_2|^2})\frac{\partial}{\partial z_1},
\]
so
\begin{equation}
\label{eq:nabla_bar_Z_L}
  \nabla_{\bar Z}L_{r_\gamma}|_{S_\gamma}=\frac{i}{\bar z_2}L_{r_\gamma},
\end{equation}
Using our given metric, we use the metric compatibility of $\nabla$ with \eqref{eq:omega_gamma_on_S} to compute
\begin{align*}
  \nabla_Z L_{r_\gamma}|_{S_\gamma}&=\frac{1}{f(\log|z_2|^2)}\left<\nabla_Z L_{r_\gamma},\frac{\partial}{\partial z_1}\right>\frac{\partial}{\partial z_1}+\frac{|z_2|^2}{s}\left<\nabla_Z L_{r_\gamma},\frac{\partial}{\partial z_2}\right>\frac{\partial}{\partial z_2}\\
  &=\frac{1}{f(\log|z_2|^2)}Z\left<L_{r_\gamma},\frac{\partial}{\partial z_1}\right>\frac{\partial}{\partial z_1}.
\end{align*}
Now \eqref{eq:omega_gamma_on_S} gives us
\begin{align*}
  Z\left<L_{r_\gamma},\frac{\partial}{\partial z_1}\right>\bigg|_{S_\gamma}&=Z\left(f(\log|z_2|^2)e^{i\log|z_2|^2}\right)\\
  &=\left(f'(\log|z_2|^2)+if(\log|z_2|^2)\right)e^{i\log|z_2|^2}\frac{1}{z_2},
\end{align*}
so
\begin{equation}
\label{eq:nabla_Z_L}
  \nabla_Z L_{r_\gamma}|_{S_\gamma}=\left(-2\tan(t^{-1}\log|z_2|^2)+i\right)\frac{1}{z_2}L_{r_\gamma}.
\end{equation}
Differentiating \eqref{eq:nabla_bar_Z_L} and \eqref{eq:nabla_Z_L} and using these formulas again to simplify the results, we have
\[
  \nabla_Z\nabla_{\bar Z}L_{r_\gamma}|_{S_\gamma}=-\left(2i\tan(t^{-1}\log|z_2|^2)+1\right)\frac{1}{|z_2|^2}L_{r_\gamma},
\]
and
\[
  \nabla_{\bar Z}\nabla_Z L_{r_\gamma}|_{S_\gamma}=-\left(2t^{-1}\sec^2(t^{-1}\log|z_2|^2)+\left(2i\tan(t^{-1}\log|z_2|^2)+1\right)\right)\frac{1}{|z_2|^2}L_{r_\gamma},
\]
so
\[
  R_\nabla(Z,\bar Z)L_{r_\gamma}|_{S_\gamma}=2t^{-1}\sec^2(t^{-1}\log|z_2|^2)\frac{1}{|z_2|^2}L_{r_\gamma}.
\]
Since we may choose $\nu_{\mathbb{C}}=|L_{r_\gamma}|^{-1}L_{r_\gamma}$, we have
\[
  \left<R_\nabla(Z,\bar Z)\nu_{\mathbb{C}},\nu_{\mathbb{C}}\right>|_{S_\gamma}=2t^{-1}\sec^2(t^{-1}\log|z_2|^2)\frac{1}{|z_2|^2}.
\]

Since $X_{r_\gamma}=\frac{1}{2}(L_{r_\gamma}+\bar L_{r_\gamma})$, \eqref{eq:nabla_bar_Z_L} and \eqref{eq:nabla_Z_L} also imply that
\[
  \nabla_Z X_{r_\gamma}|_{S_\gamma}=\left(-\tan(t^{-1}\log|z_2|^2)+\frac{i}{2}\right)\frac{1}{z_2}L_{r_\gamma}-\frac{i}{2z_2}\bar L_{r_\gamma}.
\]
We immediately obtain $\sff_\nabla(Z,Z)|_{S_\gamma}\equiv 0$, so the first sum on the left-hand side of \eqref{eq:DF_Index_boundary_geometry} must vanish, since we can take $W_1=|z_2|s^{-1/2}Z$ on $S_\gamma$.  Since $JX_{r_\gamma}=\frac{i}{2}(L_{r_\gamma}-\bar L_{r_\gamma})$, we have
\[
  \left<\nabla_Z X_{r_\gamma},JX_{r_\gamma}\right>|_{S_\gamma}=\left(i\tan(t^{-1}\log|z_2|^2)+1\right)\frac{1}{2z_2}|L_{r_\gamma}|^2.
\]
We have $|L_{r_\gamma}|^2=2|X_{r_\gamma}|^2$, so
\[
  \left<\nabla_Z \nu_{\mathbb{R}},J\nu_{\mathbb{R}}\right>|_{S_\gamma}=\left(i\tan(t^{-1}\log|z_2|^2)+1\right)\frac{1}{z_2},
\]
and hence
\[
  |\sff_\nabla(Z,J\nu_{\mathbb{R}})|^2|_{S_\gamma}=\left(\tan^2(t^{-1}\log|z_2|^2)+1\right)\frac{1}{|z_2|^2}=\sec^2(t^{-1}\log|z_2|^2)\frac{1}{|z_2|^2}.
\]
This gives us
\begin{multline*}
  \left(\frac{1}{2}\left<R_\nabla(Z,\bar Z)\nu_{\mathbb{C}},\nu_{\mathbb{C}}\right>-\frac{\eta}{1-\eta}|\sff_\nabla(Z,J\nu_{\mathbb{R}})|^2\right)\bigg|_{S_\gamma}=\\
  \left(\frac{1}{t}-\frac{\eta}{1-\eta}\right)\sec^2(t^{-1}\log|z_2|^2)\frac{1}{|z_2|^2}.
\end{multline*}
By construction, this is strictly positive, so \eqref{eq:DF_Index_boundary_geometry} must hold for some $C>0$ sufficiently small.  Hence, for any $0<\eta<DF(\Omega_\gamma)$, we have constructed a metric satisfying the requirements of Theorem \ref{thm:geometric_equivalence} for the worm domain $\Omega_\gamma$.  Indeed, any Hermitian metric satisfying \eqref{eq:omega_gamma_on_S} will satisfy \eqref{eq:DF_Index_boundary_geometry}, so Theorem \ref{thm:geometric_equivalence} suggests that such metrics may better capture the geometry of the worm domain than the traditional Euclidean metric.

As a second closely-related example, we will explicitly define a special case of the family of examples constructed by Diederich and Ohsawa in \cite{DiOh82}.  We let $M_\infty$ denote the quotient of the space $\mathbb{C}\times(\mathbb{C}\backslash\{0\})$ by the equivalence relation $(z_1,z_2)\sim(e^{i\log 4}z_1,2z_2)$.  This is a non-compact, non-Stein K\"ahler manifold with K\"ahler form given by $\omega_\infty=idz_1\wedge d\bar z_1+i\frac{1}{|z_2|^2}dz_2\wedge d\bar z_2$.  Define $\Omega_\infty\subset M_\infty$ by the real-analytic defining function $r_\infty(z)=\abs{z_1+e^{i\log|z_2|^2}}^2-1$.  The same proof used for the worm domain can be used to show that $\Omega_\infty$ is a relatively compact pseudoconvex domain, and $\partial\Omega_\infty$ is strictly pseudoconvex except on the set $S_\infty=\{(z_1,z_2)\in M_\infty:z_1=0\}$.  The equivalence relation defining $M_\infty$ guarantees that $S_\infty$ is a compact Riemann surface of genus $1$.  Since $S_\infty$ is a complex submanifold of $\partial\Omega_\infty$, Corollary \ref{cor:complex_submanifold_Kahler} implies that $DF(\Omega_\infty)=0$.  On the other hand, Diederich and Ohsawa show in \cite{DiOh82} that $\Omega_\infty$ is still Stein.

This is not the first known example of a Stein domain with vanishing Diederich-Forn{\ae}ss Index: see Section 4 of \cite{FuSh18} for an example of a Stein domain with a real-analytic, Levi-flat boundary admitting no bounded plurisubharmonic exhaustion functions.  The $L^2$ function theory for this domain is already known to be pathological, since the Cauchy-Riemann operator does not have closed range \cite{ChSh15}.  Note that Corollary \ref{cor:complex_submanifold_Kahler} implies that the Diederich-Forn{\ae}ss Index for any domain bounded by a Levi-flat hypersurface foliated by compact complex manifolds must be zero.  See \cite{DeFo20} for a general technique for constructing such manifolds and a wide range of examples.

\section{Third Derivatives in Hermitian Geometry}

\label{sec:Third_derivatives}

As noted in the introduction, derivatives of the Levi-form arise naturally in the study of the Diederich-Forn{\ae}ss Index, and these necessarily involve third derivatives of a defining function.  In this section, we will review the definition and basic properties of the Hessian in Hermitian geometry.  Then we will generalize this construction to obtain a third-order differential operator and study its symmetry properties.

The Chern connection $\nabla$ on $T^{1,0}(M)$ is the unique connection characterized by the properties that $\nabla$ is compatible with the Hermitian metric and $\nabla$ satisfies $\nabla_{\bar W}Z=0$ for all holomorphic sections $Z$ of $T^{1,0}(U)$, where $U\subset M$ is an open set and $W\in T^{1,0}(U)$.  If $Z,W\in T^{1,0}(U)$ are both holomorphic on $U\subset M$, then we have $T_\nabla(Z,\bar W)\equiv 0$.  Since $T_\nabla$ is a tensor and every vector can be locally expressed as a linear combination of holomorphic vector fields, we conclude that
\begin{equation}
\label{eq:torsion_free}
  T_\nabla(Z,\bar W)\equiv 0\text{ for all }Z,W\in T^{1,0}(M).
\end{equation}

For $X,Y\in T(M)$, we define the Hessian operator
\[
  \Hess_\nabla(X,Y)=XY-\nabla_X Y.
\]
\textit{A priori}, we must use a $C^1$ section $Y$ to define $\Hess_\nabla(X,Y)$.  If $\{V_j\}_{1\leq j\leq 2n}$ is a local basis of smooth sections for $T(M)$ on some open set $U\subset M$, then we can write $Y=\sum_{j=1}^{2n} a^j V_j$ for $\{a^j\}_{1\leq j\leq 2n}\subset C^1(U)$, and hence $\nabla_X Y=\sum_{j=1}^{2n}((X a^j)V_j+a^j\nabla_X V_j)$.  For any $f\in C^2(U)$ we have
\[
  \Hess_\nabla(X,Y)f=\sum_{j=1}^{2n} a^j\Hess_\nabla(X,V_j)f.
\]
Since this demonstrates that $\Hess_\nabla(X,Y)$ depends only on the point-wise values of $X$ and $Y$, we see that $\Hess_\nabla(X,Y)$ is well-defined for all vector fields $X,Y\in T(M)$.  The Hessian operator is not a symmetric tensor when the torsion of the connection is non-trivial, but we have
\begin{equation}
\label{eq:Hessian_symmetry}
  \Hess_\nabla(X,Y)-\Hess_\nabla(Y,X)=-T_\nabla(X,Y)\text{ for all }X,Y\in T^{1,0}(M).
\end{equation}
If we complexify so that this is $\mathbb{C}$-linear in $X$ and $Y$, the complex Hessian has an invariant meaning that is independent of the metric.  For $Z,W\in T^{1,0}(M)$ and $f\in C^2(M)$, the invariant definition of $\partial$ gives us
\[
  \ddbar f(Z,\bar W)=Z\bar W f-\dbar f([Z,\bar W]),
\]
so \eqref{eq:torsion_free} implies
\begin{equation}
\label{eq:complex_hessians_equal}
  \ddbar f(Z,\bar W)=\Hess_\nabla(Z,\bar W)f\text{ for all }Z,W\in T^{1,0}(M).
\end{equation}

We define a third-order generalization of the Hessian in the natural way:
\[
  H^3_\nabla(X_1,X_2,X_3)=X_1\Hess_\nabla(X_2,X_3)-\Hess_\nabla(\nabla_{X_1} X_2,X_3)-\Hess_\nabla(X_2,\nabla_{X_1}X_3)
\]
for all $X_1,X_2,X_3\in T(M)$.  Once again, we must initially assume that $X_2$ and $X_3$ are $C^1$ sections (we do not require $C^2$ sections since we have already shown that $\Hess_\nabla(X_2,X_3)$ only depends on the point-wise values of its arguments).  For $U$ and $\{V_j\}_{1\leq j\leq 2n}$ as before, we may write $X_1=\sum_{j=1}^{2n} a^j V_j$ and $X_2=\sum_{j=1}^{2n} b^j V_j$ for $\{a_j\}_{1\leq j\leq 2n}\subset C^1(U)$ and $\{b_j\}_{1\leq j\leq 2n}\subset C^1(U)$.  As before, we compute
\[
  H^3_\nabla(X_1,X_2,X_3)=\sum_{j,k=1}^{2n}a^j b^k H^3_\nabla(X_1,V_j,V_k),
\]
so $H^3_\nabla(X_1,X_2,X_3)$ only depends on the point-wise values of $X_2$ and $X_3$, and hence $H^3_\nabla(X_1,X_2,X_3)$ is well-defined for any vector fields $H^3_\nabla(X_1,X_3,X_3)$.  We may also complexify so that $H^3_\nabla$ is $\mathbb{C}$-linear in $X_1$, $X_2$, and $X_3$.  We summarize the symmetries of $H^3_\nabla$ in the following lemma:
\begin{lem}
\label{lem:H_3_symmetries}
  For $H^3_\nabla$ defined as above and $L,Z,W\in T^{1,0}(M)$, we have
  \begin{align}
    \label{eq:H_3_first_symmetry_unmixed}
    H^3_\nabla(L,Z,\bar W)-H^3_\nabla(Z,L,\bar W)&=-\Hess_\nabla(T_\nabla(L,Z),\bar W),\\
    \label{eq:H_3_first_symmetry_mixed}
    H^3_\nabla(\bar W,Z,L)-H^3_\nabla(Z,\bar W,L)&=-R_\nabla(\bar W,Z)L,\\
    \label{eq:H_3_second_symmetry_mixed}
    H^3_\nabla(L,Z,\bar W)-H^3_\nabla(L,\bar W,Z)&=0,\text{ and}\\
    \label{eq:H_3_third_symmetry_unmixed}
    H^3_\nabla(Z,\bar W,L)-H^3_\nabla(L,\bar W,Z)&=-\Hess_\nabla(\bar W,T_\nabla(Z,L)).
  \end{align}
\end{lem}

\begin{proof}
  Observe that each identity in Lemma \ref{lem:H_3_symmetries} depends only on the point-wise values of $L$, $Z$, and $W$, so it suffices to proves these identities for vector fields that are locally holomorphic.  Henceforth, we let $Z,W,L\in T^{1,0}(U)$ be holomorphic vector fields on some neighborhood $U$ in $M$.  For holomorphic vector fields, we have $R_\nabla(\bar W,Z)L=\nabla_{\bar W}\nabla_Z L$.  Hence, we have
  \begin{align*}
    H^3_\nabla(L,Z,\bar W)&=LZ\bar W-\Hess_\nabla(\nabla_L Z,\bar W),\\
    H^3_\nabla(L,\bar W,Z)&=L\bar W Z-\Hess_\nabla(\bar W,\nabla_L Z),\text{ and}\\
    H^3_\nabla(\bar W,Z,L)&=\bar W Z L-\Hess_\nabla(\bar W,\nabla_Z L)-R_\nabla(\bar W,Z)L.
  \end{align*}
  These formulas allow us to immediately compute \eqref{eq:H_3_first_symmetry_unmixed} and \eqref{eq:H_3_first_symmetry_mixed}.  If we use \eqref{eq:Hessian_symmetry} and \eqref{eq:torsion_free} to simplify, we obtain \eqref{eq:H_3_second_symmetry_mixed}.  Since $[Z,L]$ is also necessarily holomorphic we have $\nabla_{\bar W}[Z,L]=0$, and hence \eqref{eq:H_3_third_symmetry_unmixed} follows.
\end{proof}

One of our primary applications of Lemma \ref{lem:H_3_symmetries} will be the special case in which we have a section $X$ of $T(M)$ satisfying $X=\frac{1}{2}(L+\bar L)$ for some $L\in T^{1,0}(M)$.  Then we have
\begin{multline*}
  H^3_\nabla(Z,\bar W,X)-H^3_\nabla(X,Z,\bar W)=\\\frac{1}{2}(H^3_\nabla(Z,\bar W,L)-H^3_\nabla(L,Z,\bar W)+H^3_\nabla(Z,\bar W,\bar L)-H^3_\nabla(\bar L,Z,\bar W)),
\end{multline*}
so \eqref{eq:H_3_first_symmetry_mixed} and \eqref{eq:H_3_second_symmetry_mixed} give us
\begin{multline*}
  H^3_\nabla(Z,\bar W,X)-H^3_\nabla(X,Z,\bar W)=\\\frac{1}{2}(H^3_\nabla(Z,\bar W,L)-H^3_\nabla(L,\bar W,Z)-R_\nabla(Z,\bar W)\bar L+H^3_\nabla(\bar W,Z,\bar L)-H^3_\nabla(\bar L,Z,\bar W)),
\end{multline*}
and two applications of \eqref{eq:H_3_third_symmetry_unmixed} give us
\begin{multline}
\label{eq:H_3_symmetry_cycle}
  H^3_\nabla(Z,\bar W,X)-H^3_\nabla(X,Z,\bar W)=\\-\frac{1}{2}\left(\Hess_\nabla(\bar W,T_\nabla(Z,L))+R_\nabla(Z,\bar W)\bar L+\Hess_\nabla(Z,T_\nabla(\bar W,\bar L))\right)
\end{multline}
for all $Z,W,L\in T^{1,0}(M)$ and $X=\frac{1}{2}(L+\bar L)$.

\section{Regularity Properties of \texorpdfstring{$C^2$}{C2} Domains}

\label{sec:regularity_properties}

In this section, we will discuss the basic properties of admissible defining functions.  As noted in the introduction, our primary goal for this section is to develop the tools necessary to take derivatives of the Levi-form in directions transverse to the boundary, even though our defining functions are only $C^2$.  Although we will ultimately use these results on Hermitian manifolds, we do not need the complex structure for the results of this section, so we will work on real Riemannian manifolds.  We first define those transverse directions for which this is possible.  Note that the following definition and lemma are completely independent of the choice of metric.
\begin{defn}
\label{defn:compatibility}
  Let $M$ be a $C^3$ manifold of real dimension $n\geq 2$ and let $\Omega\subset M$ be a relatively compact domain with $C^2$ boundary.  Given a neighborhood $U$ of $\partial\Omega$, we say that a $C^1$ section $X$ of $T(U)$ and a defining function $r\in C^2(U)$ for $\Omega$ are compatible if
  \begin{enumerate}
    \item
    \begin{equation}
    \label{eq:compatible_hypothesis_Xr}
      Xr\equiv 1\text{ on }U\text{, and}
    \end{equation}
    \item
    \begin{equation}
    \label{eq:compatible_hypothesis_commutator}
      dr([X,Y])\in C^1(U)\text{ for any }C^2\text{ section }Y\text{ of }T(U).
    \end{equation}
  \end{enumerate}
\end{defn}

Our fundamental lemma for compatible defining functions and sections is the following:
\begin{lem}
\label{lem:third_derivatives_exist}
  Let $M$ be a $C^3$ manifold of real dimension $n\geq 2$ and let $\Omega\subset M$ be a relatively compact domain with $C^2$ boundary equipped with a defining function $r\in C^2(U)$ on some neighborhood $U$ of $\partial\Omega$ and a $C^1$ section $X$ of $T(U)$ that is compatible with $r$.  Then for any affine connection $\nabla$ on $T(M)$ and every pair of continuous vector fields $Y_1,Y_2\in T(U)$, $H^3_\nabla(X,Y_1,Y_2)r$ exists and is continuous on $U$.
\end{lem}

\begin{proof}
  Suppose first that $Y_1$ and $Y_2$ are $C^2$ sections.  Since $X$ is compatible with $r$, Definition \ref{defn:compatibility} implies that $X(Y_2 r)=[X,Y_2]r\in C^1(U)$.  By Lemma \ref{lem:characteristic_regularity}, $X(Y_1(Y_2 r))$ exists on $U$ and $X(Y_1(Y_2 r))=[X,Y_1](Y_2 r)+Y_1([X,Y_2]r)$ on $U$, so $X(Y_1(Y_2 r))$ must be continuous on $U$.  Hence, $H^3_\nabla(X,Y_1,Y_2)r$ exists and is continuous on $U$.  For continuous vector fields $Y_1$ and $Y_2$, we may write these as continuous linear combinations of $C^2$ sections and use the fact that $H^3_\nabla(X,Y_1,Y_2)r$ is a tensor with respect to $Y_1$ and $Y_2$ to complete the proof.
\end{proof}

To build a useful class of examples, we will need to impose a metric on $M$.  The following generalizes known results about the signed distance function (e.g., \cite{KrPa81} for the special case in which $h\equiv 0$ and $M=\mathbb{R}^n$).
\begin{lem}
\label{lem:signed_distance_function}
  Let $M$ be a $C^3$ manifold of real dimension $n\geq 2$ equipped with a Riemannian metric and let $\Omega\subset M$ be a relatively compact domain with $C^2$ boundary.  Let $h$ be a $C^2$ function on some neighborhood of $\partial\Omega$.  Then there exists a neighborhood $U$ of $\partial\Omega$ and a defining function $r\in C^2(U)$ for $\Omega$ such that
  \begin{equation}
  \label{eq:eikonal}
    |dr|\equiv e^{h}\text{ on }U,
  \end{equation}
  and such that if $X_r$ is the unique $C^1$ section $X_r$ of $T(U)$ satisfying
  \begin{equation}
  \label{eq:dual_normal}
    e^{-2h}dr(Y)=\left<Y,X_r\right>\text{ on }U\text{ for all }Y\in T(U),
  \end{equation}
  then $X_r$ is compatible with $r$ on $U$.
\end{lem}

\begin{proof}
  The existence of a unique $C^1$ defining function $r$ satisfying the eikonal equation \eqref{eq:eikonal} on some neighborhood of $\partial\Omega$ follows from the method of characteristics.  Given $X_r$ characterized by \eqref{eq:dual_normal}, let $U$ be a sufficiently small neighborhood on which Lemma \ref{lem:characteristic_regularity} holds.

  Note that \eqref{eq:dual_normal} implies that two linear operators on $T(U)$ are identical and hence must have the same norm.  This means that $e^{-2h}|dr|=|X_r|$, but then \eqref{eq:eikonal} implies $|X_r|=e^{-h}$.  Now \eqref{eq:dual_normal} with $Y=X_r$ gives us
  \[
    e^{-2h}dr(X_r)=|X_r|^2=e^{-2h},
  \]
  and hence $X_r r\equiv 1$ on $U$, so we have \eqref{eq:compatible_hypothesis_Xr}.

  Let $\rho$ be an arbitrary $C^2$ defining function for $\Omega$ on $U$.  Then $r=f\rho$ on $U$ for some positive-valued function $f\in C^1(U\backslash\partial\Omega)\cap C(U)$, and $dr=f d\rho$ on $\partial\Omega$.  Since \eqref{eq:compatible_hypothesis_Xr} implies $1=f X_r\rho$ on $\partial\Omega$, we have $f=(X_r\rho)^{-1}$ on $\partial\Omega$, and hence $f|_{\partial\Omega}\in C^1(\partial\Omega)$.  If $Y$ is a $C^1$ section of $T(U)$, then we have $Yr=fY\rho$ on $\partial\Omega$, so $Yr|_{\partial\Omega}\in C^1(\partial\Omega)$ even though $r$ itself is not yet known to be $C^2$.

  To show that $r\in C^2(U)$, we let $D$ denote the Levi-Civita connection on $T(M)$.  For a $C^2$ section $Y$ of $T(M)$, since the Levi-Civita connection is torsion-free we have
  \[
    [X_r,Y]r=(D_{X_r} Y)r-(D_Y X_r)r.
  \]
  Since $Y$ is a $C^2$ section, $(D_{X_r} Y)$ is a $C^1$ section.  On the other hand, \eqref{eq:dual_normal} implies
  \[
    (D_Y X_r)r=dr(D_Y X_r)=e^{2h}\left<D_Y X_r,X_r\right>.
  \]
  Since $|X_r|^2\equiv e^{-2h}$ and $D$ is compatible with the metric, we have
  \[
    -2e^{-2h}Yh=Y|X_r|^2=\left<D_Y X_r,X_r\right>+\left<X_r,D_Y X_r\right>,
  \]
  and hence \eqref{eq:dual_normal} coupled with symmetry of the inner product implies
  \[
    -2e^{-2h}Yh=2e^{-2h}dr(D_Y X_r).
  \]
  This implies that
  \begin{equation}
  \label{eq:D_Y_X_r}
    (D_Y X_r)r\equiv -(Yh)\text{ on }U.
  \end{equation}
  We may write
  \begin{equation}
  \label{eq:r_signed_distance_commutator_identity}
    [X_r,Y]r=((D_{X_r} Y)+(Yh)X_r)r,
  \end{equation}
  but $((D_{X_r} Y)+(Yh)X_r)$ is a $C^1$ section.  Note that Lemma \ref{lem:characteristic_regularity} implies that $X_r(Yr)$ exists on $U$ and satisfies $X_r(Y r)=[X_r,Y]r$.  Hence \eqref{eq:r_signed_distance_commutator_identity} implies
  \[
    X_r(Y r)=((D_{X_r} Y)+(Yh)X_r)r.
  \]
  If $\{Y_j\}_{j=1}^n$ is a local basis of $C^2$ sections for $T(\tilde U)$ on some neighborhood $\tilde U\subset U$, then we may set $f_j=Y_j r$ and use Lemma \ref{lem:characteristic_regularity} to show that each $Y_j r\in C^1(\tilde U)$.  This means that $r\in C^2(\tilde U)$.  Since $\tilde U$ is arbitrary, $r\in C^2(U)$.  Once we know that $r\in C^2(U)$, \eqref{eq:r_signed_distance_commutator_identity} implies \eqref{eq:compatible_hypothesis_commutator}.
\end{proof}

In fact, we have the following equivalence:
\begin{lem}
\label{lem:signed_distance_function_equivalence}
  Let $M$ be a $C^3$ manifold of real dimension $n\geq 2$ equipped with a Riemannian metric and let $\Omega\subset M$ be a relatively compact domain with $C^2$ boundary.  Let $U$ be a neighborhood of $\partial\Omega$, let $r\in C^2(U)$ be a defining function for $\Omega$, and let $X_r$ be the unique $C^1$ section of $T(U)$ satisfying
  \begin{equation}
  \label{eq:metric_compatibility}
    |dr|^{-2}dr(Y)=\left<Y,X_r\right>\text{ on }U\text{ for all }Y\in T(U).
  \end{equation}
  Then $X_r$ is compatible with $r$ on $U$ if and only if $|dr|\in C^2(U)$.
\end{lem}

\begin{proof}
  If $|dr|\in C^2(U)$, then we may set $h=\log|dr|$ and apply Lemma \ref{lem:signed_distance_function} to show that $X_r$ is compatible with $r$ on $U$.

  Suppose $X_r$ is compatible with $r$ on $U$, and let $D$ denote the Levi-Civita connection on $T(M)$.  Let $Y$ be a $C^2$ section of $T(U)$.  Since \eqref{eq:metric_compatibility} implies that two linear operators on $T(U)$ are identical, these operator must have the same norm, and hence
  \begin{equation}
  \label{eq:dual_norm_identity_real}
    |X_r|=|dr|^{-1}\text{ on }U.
  \end{equation}
  Since $D$ is metric-compatible, we compute
  \[
    Y(|dr|^{-2})=Y|X_r|^2=2\left<D_Y X_r,X_r\right>,
  \]
  but then \eqref{eq:metric_compatibility} implies
  \[
    Y(|dr|^{-2})=2|dr|^{-2}dr(D_Y X_r).
  \]
  Since $D$ is also torsion-free, we have
  \[
    Y(|dr|^{-2})=2|dr|^{-2}dr(D_{X_r} Y)-2|dr|^{-2}dr([X_r,Y]).
  \]
  The first term on the right-hand side is $C^1$ because $Y$ is a $C^2$ section, and the second term is $C^1$ by \eqref{eq:compatible_hypothesis_commutator}.  Hence, $Y(|dr|^{-2})\in C^1(U)$.  Since $Y$ was arbitrary, $|dr|^{-2}\in C^2(U)$, which implies $|dr|\in C^2(U)$.
\end{proof}

Lemma \ref{lem:signed_distance_function_equivalence} motivates our definition of admissible defining function:
\begin{defn}
\label{defn:metric_compatibility}
  Let $M$ be a $C^3$ manifold of real dimension $n\geq 2$ equipped with a Riemannian metric and let $\Omega\subset M$ be a relatively compact domain with $C^2$ boundary.  Given a neighborhood $U$ of $\partial\Omega$ and a defining function $r\in C^2(U)$ for $\Omega$, we say that $r$ is \textit{admissible} if $|dr|\in C^2(U)$.
\end{defn}
We observe that \eqref{eq:dual_normal} is equivalent to \eqref{eq:metric_compatibility}, so Lemma \ref{lem:signed_distance_function} gives us a large class of admissible defining functions $r$ and guarantees that the unique $C^1$ section $X_r$ of $T(U)$ satisfying \eqref{eq:metric_compatibility} is always compatible with $r$ in the sense of Definition \ref{defn:compatibility}.  We further observe that \eqref{eq:metric_compatibility} with $Y=X_r$ and \eqref{eq:dual_norm_identity_real} imply that
\begin{equation}
\label{eq:Xr_equals_one}
  dr(X_r)\equiv 1\text{ on }U.
\end{equation}

\begin{rem}
\label{rem:real_vs_complex}
  Suppose $M$ is a Hermitian manifold and $\Omega\subset M$ is a relatively compact domain with $C^2$ boundary with an admissible $C^2$ defining function $r$ on some neighborhood $U$ of $\partial\Omega$.  It is not difficult to check that the unique $C^1$ section $L_r$ of $T^{1,0}(U)$ satisfying \eqref{eq:metric_compatibility_complex} and the unique $C^1$ section $X_r$ of $T(U)$ satisfying \eqref{eq:metric_compatibility} are related by $L_r=X_r-iJX_r$ and $X_r=\frac{1}{2}(L_r+\bar L_r)$.  Furthermore, the same reasoning used to prove \eqref{eq:dual_norm_identity_real} and \eqref{eq:Xr_equals_one} implies
  \begin{equation}
  \label{eq:dual_norm_identity_complex}
    |L_r|=|\partial r|^{-1}\text{ on }U
  \end{equation}
  and
  \begin{equation}
  \label{eq:Lr_equals_one}
    \partial r(L_r)\equiv 1\text{ on }U.
  \end{equation}
\end{rem}

If $X_r\in T(U)$ satisfies \eqref{eq:metric_compatibility} for some admissible defining function $r$, then $\nu=|X_r|^{-1}X_r|_{\partial\Omega}$ is the unit outward normal.  Let $\nabla$ denote a connection that is compatible with the Riemannian metric on $M$.  Since $\left<X_r,Y\right>=0$ on $\partial\Omega$ whenever $Y\in T(\partial\Omega)$, the second fundamental form can be written:
\begin{equation}
\label{eq:sff_X_r}
  \sff_\nabla(Y_1,Y_2)=-|X_r|^{-2}\left<\nabla_{Y_1}X_r,Y_2\right>X_r
\end{equation}
We may use metric compatibility of $\nabla$ to obtain
\[
  \sff_\nabla(Y_1,Y_2)=|X_r|^{-2}\left<X_r,\nabla_{Y_1}Y_2\right>X_r.
\]
Since $Y_2 r\equiv 0$ on $\partial\Omega$ when $Y_2\in T(\partial\Omega)$, we may use \eqref{eq:metric_compatibility} and \eqref{eq:dual_norm_identity_real} to show
\[
  \sff_\nabla(Y_1,Y_2)=dr(\nabla_{Y_1}Y_2)X_r=-(\Hess_\nabla(Y_1,Y_2)r)X_r.
\]
In particular, we have
\begin{equation}
\label{eq:sff_identity}
  \Hess_\nabla(Y_1,Y_2)r=-|X_r|^{-2}\left<\sff_\nabla(Y_1,Y_2),X_r\right>\text{ on }\partial\Omega\text{ for all }Y_1,Y_2\in T(\partial\Omega).
\end{equation}

As one would expect, the fact that $X_r r$ is constant allows us to simplify Hessians that include $X_r$ as one of the derivatives.  We summarize these results in the following lemma:
\begin{lem}
\label{lem:symmetric_form}
  Let $M$ be a $C^3$ manifold of real dimension $n\geq 2$ equipped with a Riemannian metric and a metric-compatible connection $\nabla$.  Let $\Omega\subset M$ be a relatively compact domain with $C^2$ boundary.  Let $U$ be a neighborhood of $\partial\Omega$, let $r\in C^2(U)$ be an admissible defining function for $\Omega$, and let $X_r$ be the unique $C^1$ section of $T(U)$ satisfying \eqref{eq:metric_compatibility}.  We have
  \begin{equation}
  \label{eq:Hessian_Y_X}
    \Hess_\nabla(Y,X_r)r=|dr|^{-1}Y|dr|\text{ for all }Y\in T(U).
  \end{equation}
  For $P\in\partial\Omega$, if $\{S_j\}_{j=1}^n\subset T_P(U)$ is an orthonormal basis satisfying $S_n=\frac{X_r}{|X_r|}$, then
  \begin{equation}
  \label{eq:H_3_with_one_X_pointwise}
    H^3_\nabla(Y_2,Y_1,X_r)r|_P=|dr|^{-1}\Hess_\nabla(Y_2,Y_1)|dr|
    -\sum_{j=1}^{n-1}\left<\sff_\nabla(Y_1,S_j),\sff_\nabla(Y_2,S_j)\right>
  \end{equation}
  for all $Y_1,Y_2\in T_P(\partial\Omega)$.
\end{lem}

\begin{proof}
  Since $X_r r$ is constant, we may use \eqref{eq:metric_compatibility} to compute
  \[
    \Hess_\nabla(Y,X_r)r=-(\nabla_Y X_r)r=-|dr|^2\left<\nabla_Y X_r,X_r\right>.
  \]
  However, \eqref{eq:dual_norm_identity_real} and metric compatibility of $\nabla$ imply
  \begin{equation}
  \label{eq:derivative_of_gradient_norm}
    \left<\nabla_Y X_r,X_r\right>=\frac{1}{2}Y|X_r|^2=\frac{1}{2}Y|dr|^{-2}=-|dr|^{-3}Y|dr|,
  \end{equation}
  so \eqref{eq:Hessian_Y_X} follows.

  For $C^1$ sections $Y_1$ and $Y_2$ of $T(\partial\Omega)$, we use \eqref{eq:Hessian_Y_X} to evaluate $Y_2(\Hess_\nabla(Y_1,X_r)r)$ and $\Hess_\nabla(\nabla_{Y_2}Y_1,X_r)r$ to obtain
  \begin{multline}
  \label{eq:H_3_Y_Y_X}
    H^3_\nabla(Y_2,Y_1,X_r)r=\\
    -|dr|^{-2}(Y_2|dr|)(Y_1|dr|)+|dr|^{-1}\Hess_\nabla(Y_2,Y_1)|dr|-\Hess_\nabla(Y_1,\nabla_{Y_2}X_r)r.
  \end{multline}
  Now \eqref{eq:metric_compatibility} gives us
  \[
    \Hess_\nabla(Y_1,\nabla_{Y_2}X_r)r=Y_1(|dr|^2\left<\nabla_{Y_2}X_r,X_r\right>)-|dr|^2\left<\nabla_{Y_1}\nabla_{Y_2}X_r,X_r\right>,
  \]
  so metric compatibility of $\nabla$ implies
  \[
    \Hess_\nabla(Y_1,\nabla_{Y_2}X_r)r=2|dr|(Y_1|dr|)\left<\nabla_{Y_2}X_r,X_r\right>+|dr|^2\left<\nabla_{Y_2}X_r,\nabla_{Y_1}X_r\right>,
  \]
  and hence \eqref{eq:derivative_of_gradient_norm} gives us
  \[
    \Hess_\nabla(Y_1,\nabla_{Y_2}X_r)r=-2|dr|^{-2}(Y_1|dr|)(Y_2|dr|)+|dr|^2\left<\nabla_{Y_2}X_r,\nabla_{Y_1}X_r\right>.
  \]
  Substituting this in \eqref{eq:H_3_Y_Y_X}, we obtain
  \begin{multline}
  \label{eq:H_3_Y_Y_X_improved}
    H^3_\nabla(Y_2,Y_1,X_r)r=\\
    |dr|^{-2}(Y_2|dr|)(Y_1|dr|)+|dr|^{-1}\Hess_\nabla(Y_2,Y_1)|dr|-|dr|^2\left<\nabla_{Y_2}X_r,\nabla_{Y_1}X_r\right>.
  \end{multline}

  For $P\in\partial\Omega$, we let $\{S_j\}_{j=1}^n\subset T_P(U)$ be an orthonormal basis such that $S_n=\frac{X_r}{|X_r|}$.  Then \eqref{eq:dual_norm_identity_real} and \eqref{eq:derivative_of_gradient_norm} imply
  \[
    |dr|^2\left<\nabla_{Y_2}X_r,\nabla_{Y_1}X_r\right>=\sum_{j=1}^{n-1}|dr|^2\left<\nabla_{Y_2}X_r,S_j\right>\left<S_j,\nabla_{Y_1}X_r\right>+|dr|^{-2}(Y_2|dr|)(Y_1|dr|)
  \]
  For $1\leq j\leq n-1$, \eqref{eq:sff_X_r} gives us
  \[
    \left<\nabla_{Y}X_r,S_j\right>=-\left<\sff_\nabla(Y,S_j),X_r\right>\text{ for all }Y\in T_P(\partial\Omega).
  \]
  Since $\sff_\nabla(Y,S_j)=|X_r|^{-2}\left<\sff_\nabla(Y,S_j),X_r\right>X_r$, we have
  \[
    |dr|^2\sum_{j=1}^{n-1}\left<\nabla_{Y_2}X_r,S_j\right>\left<S_j,\nabla_{Y_1}X_r\right>=\sum_{j=1}^{n-1}\left<\sff_\nabla(Y_1,S_j),\sff_\nabla(Y_2,S_j)\right>,
  \]
  and \eqref{eq:H_3_with_one_X_pointwise} follows.
\end{proof}

\section{The One-Form \texorpdfstring{$\alpha$}{alpha}}

\label{sec:alpha}

\begin{defn}
\label{defn:alpha}
  Let $M$ be a Hermitian manifold of complex dimension $n\geq 2$ and let $\Omega\subset M$ be a relatively compact pseudoconvex domain with $C^2$ boundary.  Let $U$ be a neighborhood of $\partial\Omega$ and let $r\in C^2(U)$ be an admissible defining function for $\Omega$.  Let $L_r$ be the unique $C^1$ section of $T^{1,0}(U)$ satisfying \eqref{eq:metric_compatibility_complex}.  We define a real-valued $1$-form $\alpha_r\in\Lambda^1(U)$ by \eqref{eq:alpha_r_defn}.
\end{defn}

This form arises naturally when evaluating the complex Hessian near points of Levi-degeneracy.  For $P\in\partial\Omega$, if $Z\in\mathcal{N}_P(\partial\Omega)$ and $W\in T_P^{1,0}(U)$, \eqref{eq:Lr_equals_one} implies $W-\partial r(W)L_r\in T_P^{1,0}(\partial\Omega)$, so $\ddbar r(Z,\bar W-\dbar r(\bar W)\bar L_r)=0$, and hence we must have
\begin{equation}
\label{eq:null_space_identity}
  \ddbar r(Z,\bar W)=\alpha_r(Z)\dbar r(\bar W)\text{ at }P\in\partial\Omega\text{ for all }Z\in\mathcal{N}_P(\partial\Omega)\text{ and }W\in T_P^{1,0}(U).
\end{equation}

We can decompose $\alpha_r$ into two components: one depending on the choice of admissible defining function $r$ and one depending on the Hermitian geometry of $\partial\Omega$.  If we write $L_r=X_r-iJX_r$ (see Remark \ref{rem:real_vs_complex}), then for any $Z\in T^{1,0}(U)$ \eqref{eq:complex_hessians_equal} gives us
\[
  \alpha_r(Z)=\Hess_\nabla(Z,X_r)r+i\Hess_\nabla(Z,JX_r)r.
\]
By \eqref{eq:Hessian_Y_X} and \eqref{eq:sff_identity}, we have
\begin{equation}
\label{eq:alpha_r_geometric}
  \alpha_r(Z)=Z\log|dr|-i|X_r|^{-2}\left<\sff_\nabla(Z,JX_r),X_r\right>\text{ for all }Z\in T^{1,0}(U).
\end{equation}
Since $|X_r|^{-1}X_r|_{\partial\Omega}$ and $|X_r|^{-1}JX_r|_{\partial\Omega}$ are independent of $r$, we can see that
\[
  (\alpha_r-d\log|dr|)|_{T^{1,0}(\partial\Omega)\oplus T^{0,1}(\partial\Omega)}\text{ is independent of }r.
\]

As shown in \cite{BoSt93} (see also Lemma 5.14 in \cite{Str10}), if we restrict $\alpha_r$ to a complex submanifold in the boundary of a pseudoconvex domain with $C^3$ boundary, $\alpha_r$ is $d$-closed.  Our primary goal for this section will be to show that this result is still true for pseudoconvex domains with $C^2$ boundaries.
\begin{prop}
\label{prop:alpha_closed}
  Let $M$ be a Hermitian manifold of complex dimension $n\geq 2$ and let $\Omega\subset M$ be a relatively compact pseudoconvex domain with $C^2$ boundary.  Given an admissible defining function $r\in C^2(U)$ for $\Omega$ defined on some neighborhood $U$ of $\partial\Omega$, define $\alpha_r\in\Lambda^1(U)$ by \eqref{eq:alpha_r_defn}.  Let $S\subset\partial\Omega$ be a complex submanifold of complex dimension $m\geq 1$.  Then the pullback $\iota^*_S\alpha_r$ of $\alpha_r$ by the inclusion map $\iota_S:S\rightarrow M$ is $d_S$-closed in the weak sense, where $d_S$ denotes the exterior derivative on $S$.
\end{prop}

\begin{proof}
  Let $L_r$ be the unique $C^1$ section of $T^{1,0}(U)$ satisfying \eqref{eq:metric_compatibility_complex}, and set $X_r=\frac{1}{2}(L_r+\bar L_r)$ so that $X_r$ is the unique $C^1$ section of $T(U)$ satisfying \eqref{eq:metric_compatibility}.  Given any point $P$ in the interior of $S$, let $\tilde U\subset U$ be a neighborhood of $P$ on which there exist holomorphic coordinates $\{z_j\}_{j=1}^m\cup\{w_j\}_{j=1}^{n-m}$ such that $P=0$ in these coordinates,
  \[
    S\cap\tilde U=\{(z,w)\in\tilde U:w=0\},
  \]
  and
  \[
    \Omega\cap\tilde U=\{(z,w)\in\tilde U:\im w_{n-m}>f(z,w',\re w_{n-m})\}
  \]
  for some $f\in C^2(\mathbb{C}^{n-1}\times\mathbb{R})$ satisfying $f(0)=0$, where $w'=(w_1,\ldots,w_{n-m-1})$.  Note that $S\subset\partial\Omega$ implies that $f(z,0',0)=0$ whenever $z\in\mathbb{C}^m$ and $(z,0)\in\tilde U$.

  If we set $w_{n-m}=x_{n-m}+iy_{n-m}$, then $\rho(z,w)=f(z,w',x_{n-m})-y_{n-m}$ is also a defining function for $\Omega$ on $\tilde U$, and $\frac{\partial\rho}{\partial y_{n-m}}\equiv-1$.  On $\tilde U$, $r(z,w)=e^{h(z,w)}\rho(z,w)$ for some real-valued function $h\in C^1(\tilde U)\cap C^2(\tilde U\backslash\partial\Omega)$.  On $\partial\Omega\cap\tilde U$, we have $\partial r=e^h\partial\rho$.

  In our special coordinates, $T^{1,0}(S\cap\tilde U)=\Span\set{\frac{\partial}{\partial z_j}}_{j=1}^m$.  If $Z\in T^{1,0}(S\cap\tilde U)$ and $W\in T^{1,0}(\partial\Omega\cap\tilde U)$, then on $S\cap\tilde U$, pseudoconvexity implies that
  \[
    0\leq\ddbar r(Z+sW,\bar Z+\bar s\bar W)=2\re\left(s\ddbar r(W,\bar Z)\right)+|s|^2\ddbar r(W,\bar W)
  \]
  for all $s\in\mathbb{C}$.  Since the right-hand side of this expression achieves a minimum at $s=0$, it must have a critical point at $s=0$, and hence
  \[
    \ddbar r(W,\bar Z)=0\text{ on }S\cap\tilde U\text{ for all }Z\in T^{1,0}(S\cap\tilde U)\text{ and }W\in T^{1,0}(\partial\Omega\cap\tilde U).
  \]
  If we set $W=L_r-\left(\frac{\partial r}{\partial w_{n-m}}\right)^{-1}\frac{\partial}{\partial w_{n-m}}$, then \eqref{eq:Lr_equals_one} implies that $W\in T^{1,0}(\partial\Omega)$.  For any $Z\in T^{1,0}(S\cap\tilde U)$, we have
  \[
    0=\ddbar r\left(Z,\bar L_r-\left(\frac{\partial r}{\partial\bar w_{n-m}}\right)^{-1}\frac{\partial}{\partial\bar w_{n-m}}\right)=
    \alpha_r(Z)-\left(\frac{\partial r}{\partial\bar w_{n-m}}\right)^{-1}Z\frac{\partial r}{\partial\bar w_{n-m}}
  \]
  on $S\cap\tilde U$.  Hence
  \begin{equation}
  \label{eq:alpha_on_S}
    \alpha_r(Z)=\left(\frac{\partial r}{\partial\bar w_{n-m}}\right)^{-1}Z\frac{\partial r}{\partial\bar w_{n-m}}\text{ on }S\cap\tilde U\text{ for all }Z\in T^{1,0}(S\cap\tilde U).
  \end{equation}
  Since $\frac{\partial\rho}{\partial y_{n-m}}$ is constant, $Z\frac{\partial\rho}{\partial\bar w_{n-m}}=\frac{1}{2}Z\frac{\partial\rho}{\partial x_{n-m}}$ on $\tilde U$.  Since $\partial r=e^h\partial\rho$ on $\partial\Omega$, we can rewrite \eqref{eq:alpha_on_S} in the form
  \begin{equation}
  \label{eq:alpha_on_S_rho}
    \alpha_r(Z)=Zh+\frac{1}{2}\left(\frac{\partial\rho}{\partial\bar w_{n-m}}\right)^{-1}Z\frac{\partial\rho}{\partial x_{n-m}}\text{ on }S\cap\tilde U\text{ for all }Z\in T^{1,0}(S\cap\tilde U).
  \end{equation}

  Let $Z\in T^{1,0}(\tilde U)$ be a linear combination of $\set{\frac{\partial}{\partial z_j}}_{j=1}^m$ with constant coefficients.  Set $W=Z-(Z\rho)\left(\frac{\partial\rho}{\partial w_{n-m}}\right)^{-1}\frac{\partial}{\partial w_{n-m}}$, so that $W\in T^{1,0}(\partial\Omega\cap\tilde U)$.  On $\partial\Omega\cap\tilde U$, we compute the Levi-form
  \begin{multline}
  \label{eq:Levi_form_submanifold}
    \ddbar\rho(W,\bar W)=\ddbar\rho(Z,\bar Z)-\re\left((\bar Z\rho)\left(\frac{\partial\rho}{\partial\bar w_{n-m}}\right)^{-1}Z\frac{\partial\rho}{\partial x_{n-m}}\right)\\
    +\frac{1}{4}\abs{Z\rho}^2\abs{\frac{\partial\rho}{\partial w_{n-m}}}^{-2}\frac{\partial^2\rho}{\partial x_{n-m}^2}.
  \end{multline}
  Fix $z\in\mathbb{C}^m$ such that $(z,0)\in S\cap\tilde U$.  Choose $\tilde t>0$ such that
  \[
    (z,0',t+if(z,0',t))\in\partial\Omega\cap\tilde U
  \]
  for all $|t|\leq\tilde t$.  Since $Z\frac{\partial\rho}{\partial y_{n-m}}\equiv 0$, we have the linear approximation
  \[
    \abs{(\bar Z\rho)(z,0',t+if(z,0',t))-t\left(\bar Z\frac{\partial\rho}{\partial x_{n-m}}\right)(z,0)}\leq o(t)
  \]
  whenever $|t|\leq\tilde t$.  We may use continuity of second derivatives of $\rho$ and \eqref{eq:Levi_form_submanifold} to show that
  \begin{multline}
  \label{eq:Levi_form_submanifold_upper_bound}
    \left(\ddbar\rho(W,\bar W)\right)(z,0',t+if(z,0',t))\leq\left(\ddbar\rho(Z,\bar Z)\right)(z,0',t+if(z,0',t))\\
    -t\re\left(\left(\frac{\partial\rho}{\partial\bar w_{n-m}}\right)^{-1}\abs{Z\frac{\partial\rho}{\partial x_{n-m}}}^2\right)(z,0)+o(t).
  \end{multline}
  Note that $\re\left(\left(\frac{\partial\rho}{\partial\bar w_{n-m}}\right)^{-1}\right)=\frac{1}{2}\abs{\frac{\partial\rho}{\partial\bar w_{n-m}}}^{-2}\frac{\partial\rho}{\partial x_{n-m}}$ and
  \[
    \frac{\partial^2\rho}{\partial z_j\partial\bar z_k}(z,0',t+if(z,0',t))=\frac{\partial^2 f}{\partial z_j\partial\bar z_k}(z,0',t)=\frac{\partial^2\rho}{\partial z_j\partial\bar z_k}(z,0',t).
  \]
  Since $\Omega$ is pseudoconvex, the left-hand side of \eqref{eq:Levi_form_submanifold_upper_bound} is non-negative.  Combining these observations with \eqref{eq:alpha_on_S_rho}, we obtain
  \begin{equation}
  \label{eq:Laplace_v_lower_bound}
    \left(\ddbar\rho(Z,\bar Z)\right)(z,0',t)\geq
    2t\left(\frac{\partial\rho}{\partial x_{n-m}}\abs{\alpha_r(Z)-\partial h(Z)}^2\right)(z,0)-o(t).
  \end{equation}
  whenever $|t|\leq\tilde t$.

  Let $\chi\in C^\infty_0(S\cap\tilde U)$ be non-negative valued.  Fix $\tilde t>0$ such that $(z,0',t)\in\tilde U$ for all $z\in\supp\chi$ and $t\in\mathbb{R}$ satisfying $|t|\leq\tilde t$.  For $1\leq j,k\leq m$ and $t\in\mathbb{R}$ such that $|t|\leq\tilde t$, we define a $(1,1)$-form $\gamma_t$ on $S\cap\tilde U$ by
  \[
    \gamma_t\left(\frac{\partial}{\partial z_j},\frac{\partial}{\partial\bar z_k}\right)=\int_{S\cap\tilde U}\chi(z)\frac{\partial^2\rho}{\partial z_j\partial\bar z_k}(z,0',t) i^m dz\wedge d\bar z.
  \]
  For $1\leq j,k\leq m$, we use integration by parts to compute
  \[
    \gamma_t\left(\frac{\partial}{\partial z_j},\frac{\partial}{\partial\bar z_k}\right)=\int_{S\cap\tilde U}\frac{\partial^2\chi}{\partial z_j\partial\bar z_k}(z)\rho(z,0',t) i^m dz\wedge d\bar z.
  \]
  Hence, $\gamma_t$ is $C^2$ with respect to $t$.  Since $f(z,0)\equiv 0$ on $S\cap\tilde U$, $\gamma_0\equiv 0$.  Since \eqref{eq:Laplace_v_lower_bound} implies that
  \[
    \gamma_t(Z,\bar Z)\geq
    2t\int_{S\cap\tilde U}\chi(z)\left(\frac{\partial\rho}{\partial x_{n-m}}\abs{\alpha_r(Z)-\partial h(Z)}^2\right)(z,0) i^m dz\wedge d\bar z-o(t),
  \]
  we must have
  \[
    \frac{\partial}{\partial t}\gamma_t\left(Z,\bar Z\right)|_{t=0}=
    2\int_{S\cap\tilde U}\chi(z)\left(\frac{\partial\rho}{\partial x_{n-m}}\abs{\alpha_r(Z)-\partial h(Z)}^2\right)(z,0) i^m dz\wedge d\bar z.
  \]
  Since this holds for all non-negative valued $\chi\in C^\infty_0(S\cap\tilde U)$,
  \begin{equation}
  \label{eq:weak_second_derivative_rho_x}
    Z\bar Z\frac{\partial\rho}{\partial x_{n-m}}=
    2\frac{\partial\rho}{\partial x_{n-m}}\abs{\alpha_r(Z)-\partial h(Z)}^2
  \end{equation}
  on $S\cap\tilde U$ in the sense of distributions.  We compute
  \begin{multline*}
    Z\bar Z\im\log\frac{\partial\rho}{\partial w_{n-m}}=\frac{1}{4i}\left(\left(\frac{\partial\rho}{\partial w_{n-m}}\right)^{-1}-\left(\frac{\partial\rho}{\partial\bar w_{n-m}}\right)^{-1}\right)Z\bar Z\frac{\partial\rho}{\partial x_{n-m}}\\
    -\frac{1}{8i}\left(\left(\frac{\partial\rho}{\partial w_{n-m}}\right)^{-2}-\left(\frac{\partial\rho}{\partial\bar w_{n-m}}\right)^{-2}\right)\abs{Z\frac{\partial\rho}{\partial x_{n-m}}}^2
  \end{multline*}
  If we substitute \eqref{eq:alpha_on_S_rho} and \eqref{eq:weak_second_derivative_rho_x}, we obtain
  \[
    Z\bar Z\im\log\frac{\partial\rho}{\partial w_{n-m}}=0
  \]
  on $S\cap\tilde U$ in the sense of distributions.  However, since this holds for all constant linear combinations of $\set{\frac{\partial}{\partial z_j}}_{j=1}^m$, we conclude that $\im\log\frac{\partial\rho}{\partial w_{n-m}}$ is a pluriharmonic function on $S\cap\tilde U$, and standard elliptic theory implies that $\im\log\frac{\partial\rho}{\partial w_{n-m}}\in C^\infty(S\cap\tilde U)$.  We may assume that $\tilde U$ is sufficiently small so that there exists a pluriharmonic conjugate $f\in C^\infty(S\cap\tilde U)$, i.e., $f+i\im\log\frac{\partial\rho}{\partial w_{n-m}}$ is holomorphic on $S\cap\tilde U$.  Returning to \eqref{eq:alpha_on_S_rho}, we have
  \[
    \alpha_r(Z)=Z\left(h+\re\log\frac{\partial\rho}{\partial\bar w_{n-m}}-f\right)
  \]
  on $S\cap\tilde U$ for all $Z\in T^{1,0}(S\cap\tilde U)$, so
  \[
    \iota_S^*\alpha_r=d_S\left(h+\re\log\frac{\partial\rho}{\partial\bar w_{n-m}}-f\right)
  \]
  on $S\cap\tilde U$.  Since we have shown that $\iota_S^*\alpha_r$ is locally $d_S$-exact, we conclude that $\iota_S^*\alpha_r$ is globally $d_S$-closed in the weak sense.

\end{proof}

\section{The Two-Form \texorpdfstring{$\beta$}{beta}}

\label{sec:beta}

Following the construction in \cite{Har19}, we will introduce a continuous $2$-form $\beta_r$ with the following properties:
\begin{prop}
\label{prop:beta}
  Let $M$ be a Hermitian manifold of complex dimension $n\geq 2$ and let $\Omega\subset M$ be a relatively compact pseudoconvex domain with $C^2$ boundary.  Let $U$ be a neighborhood of $\partial\Omega$ and let $r\in C^2(U)$ be an admissible defining function for $\Omega$.  Let $X_r$ be the unique $C^1$ section of $T(U)$ satisfying \eqref{eq:metric_compatibility}.  There exists a continuous, real-valued $2$-form $\beta_r\in\Lambda^2(U)$ satisfying \eqref{eq:beta_C3_defn} on $U$ in the weak sense.  For $P\in\partial\Omega$ and $Z,W\in\mathcal{N}_P(\partial\Omega)$, we have
  \begin{equation}
  \label{eq:beta_unmixed_nullspace}
    \beta_r(Z,W)=0
  \end{equation}
  and
  \begin{multline}
  \label{eq:beta_mixed_nullspace}
    \beta_r(Z,\bar W)=-i H^3_\nabla(X_r,Z,\bar W)r-i\alpha_r(Z)\alpha_r(\bar W)+i(\Hess_\nabla(X_r,Z)r)\alpha_r(\bar W)\\
    +i\alpha_r(Z)(\Hess_\nabla(X_r,\bar W)r).
  \end{multline}.
\end{prop}

\begin{proof}
Let $L_r$ be the unique $C^1$ section of $T^{1,0}(U)$ satisfying \eqref{eq:metric_compatibility_complex}.  Let $Z$ and $W$ be $C^1$ sections of $T^{1,0}(U)$.  Although we must interpret $H^3_\nabla(Z,W,\bar L_r)$ as a distribution, weak derivatives satisfy the same commutation relations as classical derivatives, so \eqref{eq:H_3_first_symmetry_unmixed} with \eqref{eq:complex_hessians_equal} and \eqref{eq:alpha_r_defn} gives us
\[
  H^3_\nabla(Z,W,\bar L_r)r-H^3_\nabla(W,Z,\bar L_r)r=-\alpha_r(T_\nabla(Z,W)).
\]
As distributions, \eqref{eq:complex_hessians_equal} and \eqref{eq:alpha_r_defn} also give us
\[
  H^3_\nabla(Z,W,\bar L_r)r=(\nabla_Z\alpha_r)(W)-\ddbar r(W,\nabla_Z\bar L_r),
\]
so we have
\begin{equation}
\label{eq:alpha_r_covariant}
  (\nabla_Z\alpha_r)(W)-\ddbar r(W,\nabla_Z\bar L_r)-(\nabla_W\alpha_r)(Z)+\ddbar r(Z,\nabla_W\bar L_r)=-\alpha_r(T_\nabla(Z,W)).
\end{equation}
However, the invariant definition of the exterior derivative gives us the following relation between distributions:
\begin{align}
  \nonumber\partial\alpha_r(Z,W)&=Z(\alpha_r(W))-W(\alpha_r(Z))-\alpha_r([Z,W])\\
  \label{eq:alpha_r_differential}&=(\nabla_Z\alpha_r)(W)-(\nabla_W\alpha_r)(Z)+\alpha_r(T_\nabla(Z,W)).
\end{align}
If we substitute \eqref{eq:alpha_r_differential} in \eqref{eq:alpha_r_covariant}, we obtain
\[
  \partial\alpha_r(Z,W)-\ddbar r(W,\nabla_Z\bar L_r)+\ddbar r(Z,\nabla_W\bar L_r)=0.
\]
Hence, we may define
\begin{equation}
\label{eq:beta_unmixed_defn}
  \beta_r(Z,W)=-\frac{i}{2}\left(\ddbar r(W,\nabla_Z\bar L)-\ddbar r(Z,\nabla_W\bar L)\right)\text{ for all }Z,W\in T^{1,0}(U),
\end{equation}
to obtain a continuous form that satisfies \eqref{eq:beta_C3_defn} in the weak sense.  Since $\beta_r$ is real-valued, we also have $\beta_r(\bar Z,\bar W)=\overline{\beta_r(Z,W)}$.

We must now consider $\beta_r(Z,\bar W)$.  Note that \eqref{eq:torsion_free} implies that $[Z,\bar W]=\nabla_Z\bar W-\nabla_{\bar W}Z$, so we have the necessary decomposition of $[Z,\bar W]$ into its $T^{1,0}(U)$ and $T^{0,1}(U)$ components.  With this in mind, we may use the invariant definition of $\dbar$ to compute the distributional derivative
\[
  \dbar\alpha_r(\bar W,Z)=\bar W(\alpha_r(Z))-\alpha_r(\nabla_{\bar W}Z)=(\nabla_{\bar W}\alpha_r)(Z),
\]
so \eqref{eq:complex_hessians_equal} gives us
\[
  \dbar\alpha_r(\bar W,Z)=H^3_\nabla(\bar W,Z,\bar L_r)r+\ddbar r(Z,\nabla_{\bar W}\bar L_r).
\]
As noted earlier, weak derivatives satisfy the same commutation relations as classical derivatives, so \eqref{eq:H_3_second_symmetry_mixed} gives us
\[
  \dbar\alpha_r(\bar W,Z)=H^3_\nabla(\bar W,\bar L_r,Z)r+\ddbar r(Z,\nabla_{\bar W}\bar L_r).
\]
Applying \eqref{eq:H_3_first_symmetry_unmixed} with \eqref{eq:complex_hessians_equal}, we obtain
\[
  \dbar\alpha_r(\bar W,Z)=H^3_\nabla(\bar L_r,\bar W,Z)r-\ddbar r(Z,T_\nabla(\bar W,\bar L_r))+\ddbar r(Z,\nabla_{\bar W}\bar L_r).
\]
Since $\alpha_r$ is real-valued, we immediately obtain
\[
  \partial\alpha_r(Z,\bar W)=H^3_\nabla(L_r,Z,\bar W)r-\ddbar r(T_\nabla(Z,L_r),\bar W)+\ddbar r(\nabla_Z L_r,\bar W).
\]
Adding these (and using \eqref{eq:H_3_second_symmetry_mixed} with Remark \ref{rem:real_vs_complex} to simplify), we see that
\begin{multline}
\label{eq:beta_mixed_defn}
  \beta_r(Z,\bar W)=-i H^3_\nabla(X_r,Z,\bar W)r+\frac{i}{2}\ddbar r(T_\nabla(Z,L_r),\bar W)-\frac{i}{2}\ddbar r(\nabla_Z L_r,\bar W)\\
  +\frac{i}{2}\ddbar r(Z,T_\nabla(\bar W,\bar L_r))-\frac{i}{2}\ddbar r(Z,\nabla_{\bar W}\bar L_r)\text{ for all }Z,W\in T^{1,0}(U)
\end{multline}
satisfies \eqref{eq:beta_C3_defn} in the weak sense, but thanks to Lemma \ref{lem:third_derivatives_exist}, it is continuous when $r$ is an admissible $C^2$ defining function and $X_r$ is the unique $C^1$ section satisfying \eqref{eq:metric_compatibility}.

Now, suppose $Z,W\in\mathcal{N}_P(\partial\Omega)$ for some $P\in\partial\Omega$.  For any section $Y$ of $T(U)$, \eqref{eq:Lr_equals_one} also implies
\[
  0=Y(L_r r)=\Hess_\nabla(Y,L_r)r+(\nabla_Y L_r)r,
\]
and hence
\begin{equation}
\label{eq:normal_covariant_derivative}
  \partial r(\nabla_Y L_r)=-\Hess_\nabla(Y,L_r)r\text{ on }U\text{ for all }Y\in T(U).
\end{equation}
Now we may use \eqref{eq:null_space_identity} and \eqref{eq:normal_covariant_derivative} with \eqref{eq:beta_unmixed_defn} to obtain \eqref{eq:beta_unmixed_nullspace}.  Similarly, \eqref{eq:null_space_identity}, \eqref{eq:normal_covariant_derivative}, and \eqref{eq:Hessian_symmetry} can be substituted in \eqref{eq:beta_mixed_defn} to obtain
\begin{multline*}
  \beta_r(Z,\bar W)=-i H^3_\nabla(X_r,Z,\bar W)r+\frac{i}{2}(\Hess_\nabla(L_r,Z)r)\alpha_r(\bar W)\\
  +\frac{i}{2}\alpha_r(Z)(\Hess_\nabla(\bar L_r,\bar W)r).
\end{multline*}
Since $\bar L_r=2X_r-L_r$, we can use \eqref{eq:complex_hessians_equal} to rewrite this as \eqref{eq:beta_mixed_nullspace}.
\end{proof}

\begin{cor}
\label{cor:beta_on_submanifold}
  Let $M$ be a Hermitian manifold of complex dimension $n\geq 2$ and let $\Omega\subset M$ be a relatively compact pseudoconvex domain with $C^2$ boundary.  Let $r\in C^2(U)$ be an admissible defining function for $\Omega$ defined on some neighborhood $U$ of $\partial\Omega$.  Let $S\subset\partial\Omega$ be a complex submanifold of complex dimension $m\geq 1$.  Then the pullback $\iota_S^*\beta_r$ of $\beta_r$ by the inclusion map $\iota_S:S\rightarrow M$ is a continuous, real-valued $(1,1)$-form that is $d_S$-exact in the weak sense.
\end{cor}

\begin{proof}
  We see that $\iota_S^*\beta_r$ is a $(1,1)$-form because of \eqref{eq:beta_unmixed_nullspace}.  If we decompose $\iota_S^*\alpha_r$ into its $(0,1)$ and $(1,0)$ components via $\iota_S^*\alpha_r=\pi_{(1,0)}\iota_S^*\alpha_r+\pi_{(0,1)}\iota_S^*\alpha_r$, then since $\iota_S^*\beta_r$ is a $(1,1)$-form, $\partial_S\pi_{(1,0)}\iota_S^*\alpha_r=0$ and $\dbar_S\pi_{(0,1)}\iota_S^*\alpha_r=0$ in the weak sense.  With this in mind, \eqref{eq:beta_C3_defn} gives us $\iota_S^*\beta_r=d_S\left(-\frac{i}{2}\pi_{(0,1)}\iota_S^*\alpha_r+\frac{i}{2}\pi_{(1,0)}\iota_S^*\alpha_r\right)$ in the weak sense, so $\iota_S^*\beta_r$ is $d_S$-exact.
\end{proof}

Observe that $\alpha_r$ and $\beta_r$ both depend on the choice of metric via \eqref{eq:metric_compatibility_complex}, but we omit this dependence in our notation.  When $\partial\Omega$ is only $C^2$, the existence of a defining function which is admissible with respect to two different metrics is not obvious (unless the metrics are conformally equivalent, but $\alpha_r$ and $\beta_r$ are invariant under conformal changes of metric; see the proof of Theorem \ref{thm:geometric_equivalence}), so the impact of this dependence is negligible.  When $\partial\Omega$ is $C^3$, these forms are invariant when restricted to the null-space of the Levi-form, as we will now show.
\begin{lem}
\label{lem:geometric_invariance}
  Let $M$ be a complex manifold of complex dimension $n\geq 2$ and let $\Omega\subset M$ be a relatively compact pseudoconvex domain with $C^3$ boundary.  Let $r\in C^3(U)$ be a defining function for $\Omega$ defined on some neighborhood $U$ of $\partial\Omega$.  Let $\left<\cdot,\cdot\right>_0$ and $\left<\cdot,\cdot\right>_1$ be two Hermitian metrics for $M$.  Let $\alpha_r^j$ and $\beta_r^j$ denote the forms defined by \eqref{eq:alpha_r_defn} and \eqref{eq:beta_C3_defn} with respect to the metric $\left<\cdot,\cdot\right>_j$ for $j\in\{0,1\}$.  For any $P\in\partial\Omega$ and $Z,W\in\mathcal{N}_P(\partial\Omega)$, we have $\alpha_r^1(Z)=\alpha_r^0(Z)$ and $\beta_r^1(Z,\bar W)=\beta_r^0(Z,\bar W)$ at $P$.
\end{lem}

\begin{proof}
  Let $L_r^j$ be the unique $C^2$ section of $T^{1,0}(U)$ satisfying \eqref{eq:metric_compatibility_complex} with respect to $\left<\cdot,\cdot\right>_j$ for $j\in\{0,1\}$ and set $X_r^j=\frac{1}{2}(L_r^j+\bar L_r^j)$ as in Remark \ref{rem:real_vs_complex}.  Since $L_r^1-L_r^0\in T^{1,0}(\partial\Omega)$ by \eqref{eq:Lr_equals_one}, we have $\ddbar r(Z,\overline{L_r^1-L_r^0})=0$ at $P$, and hence $\alpha_r^1(Z)=\alpha_r^0(Z)$ at $P$.

  Let $Z$ be a $C^2$ section of $T^{1,0}(\partial\Omega)$ satisfying $Z|_P\in\mathcal{N}_P(\partial\Omega)$.  Then $\ddbar r(Z,\bar Z)\geq 0$ on $\partial\Omega$, and $\ddbar r(Z,\bar Z)=0$ at $P$, so $P$ is a critical point for $\ddbar r(Z,\bar Z)$ on $\partial\Omega$.  Since $X_r^1-X_r^0$ is tangential on $\partial\Omega$ by \eqref{eq:Xr_equals_one}, we have $(X_r^1-X_r^0)\ddbar r(Z,\bar Z)=0$ at $P$.  Using \eqref{eq:complex_hessians_equal}, we have
  \begin{multline*}
    H^3_{\nabla^1}(X_r^1,Z,\bar Z)r-H^3_{\nabla^0}(X_r^0,Z,\bar Z)r=\\-\ddbar r((\nabla_{X_r^1}^1-\nabla_{X_r^0}^0)Z,\bar Z)-\ddbar r(Z,(\nabla_{X_r^1}^1-\nabla_{X_r^0}^0)\bar Z)
  \end{multline*}
  at $P$, where $\nabla^j$ denotes the Chern connection with respect to $\left<\cdot,\cdot\right>_j$ for $j\in\{0,1\}$.  Now \eqref{eq:null_space_identity} gives us
  \begin{multline*}
    H^3_{\nabla^1}(X_r^1,Z,\bar Z)r-H^3_{\nabla^0}(X_r^0,Z,\bar Z)r=\\-\partial r(\nabla_{X_r^1}^1 Z)\alpha_r^1(\bar Z)+\partial r(\nabla_{X_r^0}^0Z)\alpha_r^0(\bar Z)-\alpha_r^1(Z)\dbar r(\nabla_{X_r^1}^1\bar Z)+\alpha_r^0(Z)\dbar r(\nabla_{X_r^0}^0\bar Z)
  \end{multline*}
  at $P$.  Since $Zr\equiv 0$ on $\partial\Omega$, this is equivalent to
  \begin{multline*}
    H^3_{\nabla^1}(X_r^1,Z,\bar Z)r-H^3_{\nabla^0}(X_r^0,Z,\bar Z)r=(\Hess_{\nabla^1}(X_r^1,Z)r)\alpha_r^1(\bar Z)\\
    -(\Hess_{\nabla^0}(X_r^0,Z)r)\alpha_r^0(\bar Z)+\alpha_r^1(Z)\Hess_{\nabla^1}(X_r^1,\bar Z)r-\alpha_r^0(Z)\Hess_{\nabla^0}(X_r^0,\bar Z)r.
  \end{multline*}
  If we substitute this in \eqref{eq:beta_mixed_nullspace}, we see that $\beta_r^1(Z,\bar Z)=\beta_r^0(Z,\bar Z)$ at $P$ for any $Z\in\mathcal{N}_P(\partial\Omega)$.  Using a standard polarization identity, we obtain the main result.
\end{proof}

As with $\alpha_r$ in \eqref{eq:alpha_r_geometric}, it will be helpful to decompose $\beta_r$ into a component which depends on the admissible defining function $r$ and components which depend only on the Hermitian geometry of $\partial\Omega$.  This can be obtained by combining \eqref{eq:beta_C3_defn} with \eqref{eq:alpha_r_geometric}, but a more useful formula can be obtained by restricting to the nullspace of the Levi-form and using \eqref{eq:beta_mixed_nullspace} to simplify.  Fix $P\in\partial\Omega$ and $Z\in \mathcal{N}_P(\partial\Omega)$.  Let $\{S_j\}_{j=1}^{2n}\subset T_P(U)$ be an orthonormal basis such that $S_{2n}=\frac{X_r}{|X_r|}$ and $S_{2j}=JS_{2j-1}$ for all $1\leq j\leq n$.  Then we may define an orthonormal basis $\{W_j\}_{j=1}^{n}\subset T_P^{1,0}(U)$ by setting $W_j=\frac{1}{\sqrt{2}}(S_{2j}+iS_{2j-1})$ for each $1\leq j\leq n$.  Observe that $W_n=\frac{L_r}{|L_r|}$.  By \eqref{eq:H_3_with_one_X_pointwise} and \eqref{eq:complex_hessians_equal}, we have
\[
  H^3_\nabla(Z,\bar Z,X_r)r|_P=|dr|^{-1}(\ddbar|dr|)(Z,\bar Z)
  -\sum_{j=1}^{2n-1}\abs{\sff_\nabla(Z,S_j)}^2.
\]
Since $S_{2n-1}=-JS_{2n}=-\frac{JX_r}{|X_r|}$, \eqref{eq:alpha_r_geometric} gives us
\[
  \abs{\alpha_r(Z)-Z\log|dr|}=\abs{\sff_\nabla(Z,S_{2n-1})},
\]
so
\begin{multline*}
  H^3_\nabla(Z,\bar Z,X_r)r|_P=\\
  |dr|^{-1}(\ddbar|dr|)(Z,\bar Z)
  -\sum_{j=1}^{2n-2}\abs{\sff_\nabla(Z,S_j)}^2-\abs{\alpha_r(Z)-Z\log|dr|}^2.
\end{multline*}
For $1\leq j\leq n-1$, $S_{2j}=\frac{1}{\sqrt{2}}(W_j+\bar W_j)$ and $S_{2j-1}=-\frac{i}{\sqrt{2}}(W_j-\bar W_j)$.  For $Z\in\mathcal{N}_P(\partial\Omega)$, \eqref{eq:sff_identity} and \eqref{eq:complex_hessians_equal} imply that $\sff_\nabla(Z,\bar W_j)|_P=0$ for $1\leq j\leq n-1$, so we have
\[
  \sff_\nabla(Z,S_{2j})|_P=\frac{1}{\sqrt{2}}\sff_\nabla(Z,W_j)
\]
and
\[
  \sff_\nabla(Z,S_{2j-1})|_P=-\frac{i}{\sqrt{2}}\sff_\nabla(Z,W_j).
\]
Hence,
\begin{multline*}
  H^3_\nabla(Z,\bar Z,X_r)r|_P=\\
  |dr|^{-1}(\ddbar|dr|)(Z,\bar Z)
  -\sum_{j=1}^{n-1}\abs{\sff_\nabla(Z,W_j)}^2-\abs{\alpha_r(Z)-Z\log|dr|}^2.
\end{multline*}
We can compute
\[
  (\ddbar\log|dr|)(Z,\bar Z)=|dr|^{-1}(\ddbar|dr|)(Z,\bar Z)-\abs{Z\log|dr|}^2,
\]
so
\begin{multline}
\label{eq:beta_r_geometric_before_symmetry}
  H^3_\nabla(Z,\bar Z,X_r)r|_P=\\
  (\ddbar\log|dr|)(Z,\bar Z)
  -\sum_{j=1}^{n-1}\abs{\sff_\nabla(Z,W_j)}^2-\abs{\alpha_r(Z)}^2+2\re(\alpha_r(Z)\bar Z\log|dr|).
\end{multline}

To relate \eqref{eq:beta_r_geometric_before_symmetry} to \eqref{eq:beta_mixed_nullspace}, we substitute \eqref{eq:beta_r_geometric_before_symmetry} into \eqref{eq:H_3_symmetry_cycle} and use \eqref{eq:complex_hessians_equal} to simplify and obtain
\begin{multline*}
  H^3_\nabla(X_r,Z,\bar Z)r|_P=
  (\ddbar\log|dr|)(Z,\bar Z)
  -\sum_{j=1}^{n-1}\abs{\sff_\nabla(Z,W_j)}^2-\abs{\alpha_r(Z)}^2\\
  +2\re(\alpha_r(Z)\bar Z\log|dr|)+\re\left(\ddbar r(T_\nabla(Z,L_r),\bar Z)\right)+\frac{1}{2}(R_\nabla(Z,\bar Z)\bar L_r)r.
\end{multline*}
If we set $\nu_{\mathbb{C}}=\frac{L_r}{|L_r|}$, then \eqref{eq:metric_compatibility_complex} and \eqref{eq:dual_norm_identity_complex} imply that
\[
  (R_\nabla(Z,\bar Z)\bar L_r)r=\left<R_\nabla(Z,\bar Z)\bar \nu_{\mathbb{C}},\bar \nu_{\mathbb{C}}\right>=-\left<R_\nabla(Z,\bar Z)\nu_{\mathbb{C}},\nu_{\mathbb{C}}\right>,
\]
where the second equality follows from the usual symmetries of the curvature tensor for a metric-compatible connection.  On the other hand, \eqref{eq:Hessian_Y_X}, \eqref{eq:Hessian_symmetry}, \eqref{eq:torsion_free}, and \eqref{eq:dual_norm_identity_complex} imply that
\[
  2\re(\alpha_r(Z)\bar Z\log|dr|)+\re\left(\ddbar r(T_\nabla(Z,L_r),\bar Z)\right)=2\re(\alpha_r(Z)\Hess_\nabla(X_r,\bar Z)r),
\]
so we have
\begin{multline*}
  H^3_\nabla(X_r,Z,\bar Z)r|_P=
  (\ddbar\log|dr|)(Z,\bar Z)
  -\sum_{j=1}^{n-1}\abs{\sff_\nabla(Z,W_j)}^2-\abs{\alpha_r(Z)}^2\\
  +2\re(\alpha_r(Z)\Hess_\nabla(X_r,\bar Z)r)-\frac{1}{2}\left<R_\nabla(Z,\bar Z)\nu_{\mathbb{C}},\nu_{\mathbb{C}}\right>.
\end{multline*}
If we substitute this in \eqref{eq:beta_mixed_nullspace}, we have
\begin{multline}
\label{eq:beta_r_geometric}
  \beta_r(Z,\bar Z)|_P=-i(\ddbar\log|dr|)(Z,\bar Z)
  +i\sum_{j=1}^{n-1}\abs{\sff_\nabla(Z,W_j)}^2\\+\frac{i}{2}\left<R_\nabla(Z,\bar Z)\nu_{\mathbb{C}},\nu_{\mathbb{C}}\right>\text{ for all }Z\in\mathcal{N}_P(\partial\Omega).
\end{multline}

\section{Levi-forms of Level Curves}

\label{sec:level_curves}

We now consider how the Levi-forms of level curves of some admissible defining function $r$ relate to the Levi-form on the boundary.  It will be helpful to have a canonical way to extend sections from the boundary into the interior.  The following is an adaptation of parallel transport that preserves \eqref{eq:tangential_level_curves} in a natural way.
\begin{lem}
  Let $M$ be a Hermitian manifold of complex dimension $n\geq 2$ and let $\Omega\subset M$ be a relatively compact domain with $C^2$ boundary.  Let $U$ be a neighborhood of $\partial\Omega$ and let $r\in C^2(U)$ be an admissible defining function for $\Omega$.  Let $X_r$ be the unique $C^1$ section of $T(U)$ satisfying \eqref{eq:metric_compatibility} and let $L_r$ be the unique $C^1$ section of $T^{1,0}(U)$ satisfying \eqref{eq:metric_compatibility_complex}.  For every $C^1$ section $\tilde Z$ of $T^{1,0}(\partial\Omega)$, there exists a unique $C^1$ section $Z$ of $T^{1,0}(U)$ such that
  \begin{enumerate}
    \item $Z|_{\partial\Omega}=\tilde Z$,
    \item
    \begin{equation}
    \label{eq:tangential_level_curves}
      \partial r(Z)\equiv 0\text{ on }U,
    \end{equation}
    and
    \item
    \begin{equation}
    \label{eq:good_coordinate_ODE}
      \nabla_{X_r} Z=-(\Hess_\nabla(X_r,Z)r)L_r\text{ on }U.
    \end{equation}
  \end{enumerate}
\end{lem}

\begin{proof}
  Let $\{z_j\}_{j=1}^n$ be a local holomorphic coordinate patch on some neighborhood $\tilde U\subset U$.  If we write $Z=\sum_{j=1}^nf_j\frac{\partial}{\partial z_j}$, then we see that \eqref{eq:good_coordinate_ODE} is equivalent to the system of ordinary differential equations
  \[
    X_r f_j+dz_j\left(\sum_{k=1}^n f_k\nabla_{X_r}\frac{\partial}{\partial z_j}\right)=-\sum_{k=1}^nf_k\left(\Hess_\nabla\left(X_r,\frac{\partial}{\partial z_k}\right)r\right)dz_j(L_r)
  \]
  for each $1\leq j\leq n$.  Note that \eqref{eq:Xr_equals_one} implies that
  \[
    \Hess_\nabla\left(X_r,\frac{\partial}{\partial z_k}\right)r=\left[X_r,\frac{\partial}{\partial z_k}\right]r-\left(\nabla_{X_r}\frac{\partial}{\partial z_k}\right)r,
  \]
  so \eqref{eq:compatible_hypothesis_commutator} implies that this coefficient is in $C^1(U)$ for every $1\leq k\leq n$.  Hence, we may apply Lemma \ref{lem:characteristic_regularity} locally to obtain a local section $Z$ satisfying the necessary hypotheses.  Since $Z$ is uniquely determined by a boundary value problem, we may patch these local sections together to obtain a $C^1$ section $Z$ of $T^{1,0}(U)$.  To see \eqref{eq:tangential_level_curves}, we first note $\partial r(\tilde Z)=0$ on $\partial\Omega$ whenever $\tilde Z\in T^{1,0}(\partial\Omega)$, and \eqref{eq:good_coordinate_ODE} and \eqref{eq:Lr_equals_one} imply
  \[
    X_r(\partial r(Z))=\Hess_\nabla(X_r,Z)r+(\nabla_{X_r} Z)r=0
  \]
  on $U$.  Hence, \eqref{eq:tangential_level_curves} holds by the uniqueness of solutions to this boundary value problem.
\end{proof}

With these preliminaries in place, we will see that the relationship between the Levi-form of a level curve of $\partial\Omega$ and the Levi-form on $\partial\Omega$ is completely characterized by $\beta_r$ and $\alpha_r$.
\begin{lem}
\label{lem:Levi_form_bounds}
  Let $M$ be a Hermitian manifold of complex dimension $n\geq 2$ and let $\Omega\subset M$ be a relatively compact pseudoconvex domain with $C^2$ boundary.  Let $U$ be a neighborhood of $\partial\Omega$ and let $r\in C^2(U)$ be an admissible defining function for $\Omega$.  Let $X_r$ be the unique $C^1$ section of $T(U)$ satisfying \eqref{eq:metric_compatibility}.  Let $\pi:U\rightarrow\partial\Omega$ be the unique $C^1$ projection satisfying $\pi(P)=P$ for all $P\in\partial\Omega$ and $X_r\pi\equiv 0$ on $U$.  For every $\epsilon>0$, there exists $\delta>0$ such that for every $C^1$ section $Z$ of $T^{1,0}(U)$ satisfying \eqref{eq:tangential_level_curves} and \eqref{eq:good_coordinate_ODE} and $z\in U\cap\Omega$ satisfying $0<-r(z)<\delta$, we have
  \begin{multline}
  \label{eq:Levi_form_lower_bound}
    \ddbar r(Z,\bar Z)|_z\geq\\(1-\epsilon)\ddbar r(Z,\bar Z)|_{\pi(z)}+r(z)\left(i\beta_r(Z,\bar Z)-|\alpha_r(Z)|^2\right)|_z-\epsilon|Z|^2|_z(-r(z))
  \end{multline}
  and
  \begin{multline}
  \label{eq:Levi_form_upper_bound}
    \ddbar r(Z,\bar Z)|_z\leq\\(1+\epsilon)\ddbar r(Z,\bar Z)|_{\pi(z)}+r(z)\left(i\beta_r(Z,\bar Z)-|\alpha_r(Z)|^2\right)|_z+\epsilon|Z|^2|_z(-r(z)).
  \end{multline}
\end{lem}

\begin{proof}
  Let $L_r$ be the unique $C^1$ section satisfying \eqref{eq:metric_compatibility_complex}.  Let $Z$ be a $C^1$ section of $T^{1,0}(U)$ satisfying \eqref{eq:tangential_level_curves} and \eqref{eq:good_coordinate_ODE}.  Then \eqref{eq:complex_hessians_equal} can be used to show
  \[
    X_r(\ddbar r(Z,\bar Z))=H^3_\nabla(X_r,Z,\bar Z)r+\ddbar r(\nabla_{X_r} Z,\bar Z)+\ddbar r(Z,\nabla_{X_r}\bar Z),
  \]
  so \eqref{eq:good_coordinate_ODE} implies
  \begin{equation}
  \label{eq:X_Hessian_rho}
    X_r(\ddbar r(Z,\bar Z))=
    H^3_\nabla(X_r,Z,\bar Z)r-2\re\left((\Hess_\nabla(X_r,Z)r)\alpha_r(\bar Z)\right).
  \end{equation}

  For $P\in\partial\Omega$ and $t\in\mathbb{R}$ sufficiently small, let $\psi(P,t)$ denote the unique map satisfying $\psi(P,0)=P$ and $\frac{\partial}{\partial t}\psi(P,t)=X_r$.  For any $z\in U$, we may apply the Fundamental Theorem of Calculus along the integral curves of $X_r$ to obtain
  \begin{equation}
  \label{eq:Levi_form_FTC}
    \ddbar r(Z,\bar Z)|_z-\ddbar r(Z,\bar Z)|_{\pi(z)}=\int_0^{r(z)}X_r(\ddbar r(Z,\bar Z))(\psi(\pi(z),t))dt.
  \end{equation}
  Observe that $X_r(\ddbar r(Z,\bar Z))$ depends only on the first derivatives of $Z$ (more precisely, first derivatives of the coefficients of $Z$ in any holomorphic coordinate system).  Since $r$ is admissible, a standard density result can be combined with Lemma \ref{lem:third_derivatives_exist} to see that $X_r(\ddbar r(Z,\bar Z))$ exists and is continuous on $U$.

  We prove \eqref{eq:Levi_form_lower_bound} by contradiction.  Suppose, on the contrary, that there exists $\epsilon>0$ such that for every $j\in\mathbb{N}$ there exists a $C^1$ section $Z_j$ of $T^{1,0}(U)$ satisfying \eqref{eq:tangential_level_curves} and \eqref{eq:good_coordinate_ODE} and $z_j\in U\cap\Omega$ satisfying $0<-r(z_j)<\frac{1}{j}$ such that \eqref{eq:Levi_form_lower_bound} fails.  Using \eqref{eq:good_coordinate_ODE}, we have
  \[
    X_r|Z_j|^2=2\re\left<\nabla_{X_r} Z_j,Z_j\right>=-2\re\left<(\Hess_\nabla(X_r,Z_j)r)L_r,Z_j\right>,
  \]
  so \eqref{eq:metric_compatibility_complex} and \eqref{eq:tangential_level_curves} imply that $X_r|Z_j|^2\equiv 0$.  Since this means that $|Z_j|^2$ is constant along level curves of $X_r$, $|Z_j||_{z_j}=|Z_j||_{\pi(z_j)}$ for all $j$.  If $Z_j|_{z_j}=Z_j|_{\pi(z_j)}=0$, then \eqref{eq:Levi_form_lower_bound} holds trivially, so these must both be non-trivial.  Hence we may normalize $Z_j$ on a neighborhood of $z_j$ and $\pi(z_j)$ without impacting \eqref{eq:Levi_form_lower_bound}, so we henceforth assume that $|Z_j|\equiv 1$ on a neighborhood of $z_j$ and $\pi(z_j)$.  After restricting to a subsequence, we may assume that $z_j\rightarrow P$ for some $P\in\partial\Omega$ and $Z_j|_{z_j}\rightarrow Z$ for some $Z\in T^{1,0}_P(\partial\Omega)$.  We necessarily have $Z_j|_{\pi(z_j)}\rightarrow Z$ as well.  The failure of \eqref{eq:Levi_form_lower_bound} implies that $\ddbar r(Z,\bar Z)|_P\leq(1-\epsilon)\ddbar r(Z,\bar Z)|_P$, but pseudoconvexity implies $\ddbar r(Z,\bar Z)|_P\geq 0$, so we must have $Z\in\mathcal{N}_P(\partial\Omega)$.

  Since pseudoconvexity implies that $\epsilon\ddbar r(Z_j,\bar Z_j)|_{\pi(z_j)}\geq 0$, the failure of \eqref{eq:Levi_form_lower_bound} at each $z_j$ implies that
  \begin{multline}
  \label{eq:levi_form_difference_quotient}
    \ddbar r(Z_j,\bar Z_j)|_{z_j}-\ddbar r(Z_j,\bar Z_j)|_{\pi(z_j)}\leq\\r(z_j)\left(i\beta_r(Z_j,\bar Z_j)-|\alpha_r(Z_j)|^2\right)|_{z_j}-\epsilon(-r(z_j))
  \end{multline}
  Observe that \eqref{eq:X_Hessian_rho} implies that $X_r(\ddbar r(Z_j,\bar Z_j))|_{\varphi(\pi(z_j),t)}\rightarrow X_r(\ddbar r(Z,\bar Z))|_P$ for all $0>t>r(z_j)$.  If we use \eqref{eq:Levi_form_FTC} to evaluate the left-hand side of \eqref{eq:levi_form_difference_quotient}, divide by $(-r(z_j))$, and take the limit (using continuity of $\alpha_r$ and $\beta_r$), we obtain
  \[
    -X_r(\ddbar r(Z,\bar Z))|_P\leq-\left(i\beta_r(Z,\bar Z)-|\alpha_r(Z)|^2\right)|_{P}-\epsilon.
  \]
  Substituting \eqref{eq:X_Hessian_rho}, we have
  \begin{multline*}
    -H^3_\nabla(X_r,Z,\bar Z)r|_P+2\re\left((\Hess_\nabla(X_r,Z)r)\alpha_r(\bar Z)|_P\right)\leq\\
    -\left(i\beta_r(Z,\bar Z)-|\alpha_r(Z)|^2\right)|_{P}-\epsilon.
  \end{multline*}
  Since $Z\in\mathcal{N}_P(\partial\Omega)$, \eqref{eq:beta_mixed_nullspace} implies $0\leq-\epsilon$, a contradiction.  Hence for every $\epsilon>0$ there exists $\delta>0$ such that \eqref{eq:Levi_form_lower_bound} holds.

  The proof of \eqref{eq:Levi_form_upper_bound} is the same.
\end{proof}

\section{The Diederich Forn{\ae}ss Index: Necessary Conditions}

\label{sec:DF_index_necessary}

Let $M$ be a Hermitian manifold of complex dimension $n\geq 2$ and let $\Omega\subset M$ be a relatively compact pseudoconvex domain with $C^2$ boundary.  Suppose that $\rho$ and $r$ are both $C^2$ defining function for $\Omega$ on some neighborhood $U$ of $\partial\Omega$.  There exists a unique, real-valued function $h\in C^1(U)\cup C^2(U\backslash\partial\Omega)$ such that $\rho=e^{-h}r$.  In fact, $h$ satisfies the stronger condition
\begin{multline}
\label{eq:second_derivative_h_limit}
  \lim_{z\rightarrow\partial\Omega}(-r(z))\Hess_\nabla(Y_1,Y_2)h(z)=0\text{ uniformly on }U\\
  \text{ for all continuous sections }Y_1\text{ and }Y_2\text{ of }T(U).
\end{multline}
We note that \eqref{eq:second_derivative_h_limit} follows from Lemma 4.1 in \cite{Har22}, which easily generalizes to the present context since Lemma 4.1 is essentially a local result.  For $0<\eta<1$, on $U\cap\Omega$ we have
\[
  \ddbar(-(-\rho)^\eta)=\ddbar(-e^{-\eta h}(-r)^\eta)=\partial(\eta e^{-\eta h}(-r)^{\eta-1}\dbar r+\eta e^{-\eta h}(-r)^\eta\dbar h),
\]
so
\begin{multline}
\label{eq:rho_eta_hessian_identity}
  \eta^{-1}(-\rho)^{-\eta}\ddbar(-(-\rho)^\eta)=(-r)^{-1}\ddbar r+(1-\eta)(-r)^{-2}\partial r\wedge\dbar r\\-\eta (-r)^{-1}\partial h\wedge\dbar r
  -\eta (-r)^{-1}\partial r\wedge\dbar h-\eta \partial h\wedge\dbar h+\ddbar h.
\end{multline}
We also have
\begin{equation}
\label{eq:rho_zero_hessian_identity}
  \ddbar(-\log(-\rho))=(-r)^{-1}\ddbar r+(-r)^{-2}\partial r\wedge\dbar r+\ddbar h
\end{equation}
on $U\cap\Omega$.  This can be thought of as the limiting case of \eqref{eq:rho_eta_hessian_identity}, since $\log x=\lim_{\eta\rightarrow 0^+}\frac{x^\eta-1}{\eta}$ for any $x>0$.

\begin{lem}
\label{lem:DF_Index_necessary_condition}
  Let $M$ be a Hermitian manifold of complex dimension $n\geq 2$ and let $\Omega\subset M$ be a relatively compact pseudoconvex domain with $C^2$ boundary such that for some $0\leq\eta<1$ there exists a $C^2$ defining function $\rho$ for $\Omega$ and a constant $C>0$ satisfying \eqref{eq:DF_Index} (for $\eta>0$) or \eqref{eq:Strong_Oka_Property} (for $\eta=0$) on $U\cap\Omega$ for some neighborhood $U$ of $\partial\Omega$.  Let $r\in C^2(U)$ be an admissible defining function for $\Omega$.  Let $X_r$ be the unique $C^1$ section of $T(U)$ satisfying \eqref{eq:metric_compatibility}, and let $\pi:U\rightarrow\partial\Omega$ be the unique $C^1$ projection satisfying $\pi(P)=P$ for all $P\in\partial\Omega$ and $X_r\pi\equiv 0$ on $U$.  Then there exists a real-valued function $h\in C^1(U)\cap C^2(U\backslash\partial\Omega)$ satisfying \eqref{eq:second_derivative_h_limit} with the property that for every $\epsilon>0$, there exists $\delta>0$ such that for every $C^1$ section $Z$ of $T^{1,0}(U)$ satisfying \eqref{eq:tangential_level_curves} and \eqref{eq:good_coordinate_ODE} and $z\in U\cap\Omega$ satisfying $0<-r(z)<\delta$, we have
\begin{multline}
\label{eq:DF_index_necessary_condition}
  (1+\epsilon)(-r(z))^{-1}\ddbar r(Z,\bar Z)|_{\pi(z)}-i\beta_r(Z,\bar Z)|_z+\ddbar h(Z,\bar Z)|_z\\
  \geq\frac{\eta}{1-\eta}\abs{\partial h(Z)-\alpha_r(Z)}^2|_z+\left(C-\epsilon\right)|Z|^2|_z.
\end{multline}
\end{lem}

\begin{proof}
Let $h\in C^1(U)\cup C^2(U\backslash\partial\Omega)$ be the unique real-valued function such that $\rho=e^{-h}r$.  By \eqref{eq:rho_eta_hessian_identity} (resp. \eqref{eq:rho_zero_hessian_identity}), \eqref{eq:DF_Index} (resp. \eqref{eq:Strong_Oka_Property}) on $\Omega\cap U$ is equivalent to
\begin{multline}
\label{eq:hessian_inequality_expanded}
  i\ddbar r+i(1-\eta)(-r)^{-1}\partial r\wedge\dbar r-i\eta\partial h\wedge\dbar r-i\eta\partial r\wedge\dbar h-i\eta(-r)\partial h\wedge\dbar h+i(-r)\ddbar h\geq\\
  C(-r)\omega.
\end{multline}

Let $L_r$ be the unique $C^1$ section of $T^{1,0}(U)$ satisfying \eqref{eq:metric_compatibility_complex}.  By \eqref{eq:metric_compatibility_complex} and \eqref{eq:tangential_level_curves}, $\left<L_r,Z\right>\equiv 0$ on $U$.  When $Z\neq 0$, set
\[
  \tilde Z=Z-\frac{(-r)}{1-\eta}\left(\alpha_r(Z)-\eta\partial h(Z)\right)L_r.
\]
Since $\left<Z,L_r\right>\equiv 0$, we have
\[
  \abs{|\tilde Z|^2-|Z|^2}\leq O((-r)^2).
\]
By \eqref{eq:Lr_equals_one} and \eqref{eq:tangential_level_curves}, we have
\[
  \partial r(\tilde Z)=-\frac{(-r)}{1-\eta}\left(\alpha_r(Z)-\eta\partial h(Z)\right).
\]
Using \eqref{eq:alpha_r_defn}, we have
\[
  \abs{\ddbar r(\tilde Z,\overline{\tilde Z})-\ddbar r(Z,\bar Z)+\frac{2(-r)}{1-\eta}\re\left(\abs{\alpha_r(Z)}^2-\eta\partial h(Z)\alpha_r(\bar Z)\right)}
  \leq O((-r)^2).
\]
Finally, we consider derivatives of $h$.  We have
\[
  |\partial h(\tilde Z)-\partial h(Z)|\leq O(-r),
\]
and \eqref{eq:second_derivative_h_limit} with \eqref{eq:complex_hessians_equal} implies that
\[
  \lim_{z\rightarrow\partial\Omega}\abs{\ddbar h(\tilde Z,\overline{\tilde Z})-\ddbar h(Z,\bar Z)}(z)=0\text{ uniformly on }U.
\]
If we apply \eqref{eq:hessian_inequality_expanded} to the pair $(\tilde Z,\overline{\tilde Z})$, we may combine this with the above inequalities to obtain
\begin{multline*}
  \ddbar r(Z,\bar Z)-\frac{\eta(-r)}{1-\eta}\abs{\partial h(Z)-\alpha_r(Z)}^2\\
  +(-r)\ddbar h(Z,\bar Z)-(-r)\abs{\alpha_r(Z)}^2\geq C(-r)|Z|^2-|Z|^2o(-r)\text{ on }\Omega\cap U.
\end{multline*}
If we combine this with \eqref{eq:Levi_form_upper_bound} from Lemma \ref{lem:Levi_form_bounds}, we obtain \eqref{eq:DF_index_necessary_condition}.
\end{proof}

\begin{prop}
\label{prop:DF_Index_necessary_condition_boundary}
  Let $M$ be a Hermitian manifold of complex dimension $n\geq 2$ and let $\Omega\subset M$ be a relatively compact pseudoconvex domain with $C^2$ boundary such that for some $0\leq\eta<1$ there exists a $C^2$ defining function $\rho$ for $\Omega$ and a constant $C>0$ such that \eqref{eq:DF_Index} (for $\eta>0$) or \eqref{eq:Strong_Oka_Property} (for $\eta=0$) holds.  Let $U$ be a neighborhood of $\partial\Omega$ and let $r\in C^2(U)$ be an admissible defining function for $\Omega$.  Then for every $0<\tilde C<C$ there exists a real-valued function $h\in C^2(U)$ satisfying
  \begin{equation}
  \label{eq:DF_Index_necessary_condition_boundary}
    -i\beta_r(Z,\bar Z)+\ddbar h(Z,\bar Z)\geq\frac{\eta}{1-\eta}\abs{\partial h(Z)-\alpha_r(Z)}^2+\tilde C|Z|^2
  \end{equation}
  on $\partial\Omega$ for all $C^1$ sections $Z$ of $T^{1,0}(\partial\Omega)$.
\end{prop}

\begin{proof}
  Fix $0<\tilde C<C$, and let $\epsilon=\frac{C-\tilde C}{3}$.  Then $0<\epsilon<\frac{C}{3}$.  Let $X_r$ be the unique $C^1$ section of $T(U)$ satisfying \eqref{eq:metric_compatibility}, and let $L_r$ be the unique $C^1$ section of $T^{1,0}(U)$ satisfying \eqref{eq:metric_compatibility_complex}.  Let $\pi:U\rightarrow\partial\Omega$ be the unique $C^1$ projection satisfying $\pi(P)=P$ for all $P\in\partial\Omega$ and $X\pi\equiv 0$ on $U$. Let $\psi(z,t)$ denote the flow of the vector field $X_r$, so that \eqref{eq:Xr_equals_one} implies $r(\psi(P,t))=t$ and $\pi(\psi(P,t))=z$ for $P\in\partial\Omega$ and $t$ sufficiently small.  Let $h\in C^1(U)\cap C^2(U\backslash\partial\Omega)$ and $\delta>0$ be given by Lemma \ref{lem:DF_Index_necessary_condition}.

  Let $Z$ be a $C^1$ section of $T^{1,0}(U)$ satisfying \eqref{eq:tangential_level_curves} and \eqref{eq:good_coordinate_ODE}.  Let $\nabla_Z^b\bar Z=\nabla_Z\bar Z-\dbar r(\nabla_Z\bar Z)X_r$, so that \eqref{eq:Xr_equals_one} implies $(\nabla_Z^b\bar Z) r\equiv 0$ on $U$ and $\nabla_Z^b\bar Z|_{\partial\Omega}\in T^{1,0}(\partial\Omega)$.  Let $\Hess_\nabla^b(Z,\bar Z)=Z\bar Z-\nabla_Z^b\bar Z$.  By \eqref{eq:tangential_level_curves}, we have
  \[
    \nabla_Z^b\bar Z-\nabla_Z\bar Z=(\Hess_\nabla(Z,\bar Z)r)X_r,
  \]
  so \eqref{eq:complex_hessians_equal} implies that
  \begin{equation}
  \label{eq:boundary_hessian_versus_hessian}
    \Hess_\nabla^b(Z,\bar Z)h-\ddbar h(Z,\bar Z)=-(\ddbar r(Z,\bar Z))X_rh\text{ on }U\backslash\partial\Omega.
  \end{equation}

  For any $\chi\in C^\infty_0(\partial\Omega)$, we define the distribution pairing $\left<\Hess_\nabla^b(Z,\bar Z)h,\chi\right>$ on $\partial\Omega$ by
  \[
    \left<\Hess_\nabla^b(Z,\bar Z)h,\chi\right>=\int_{\partial\Omega}(\bar Z h)(\bar Z^*\bar\chi) d\sigma-\int_{\partial\Omega}((\nabla_Z^b\bar Z) h)\bar\chi d\sigma,
  \]
  where $Z^*$ denotes the adjoint of $Z$ with respect to the $L^2$ inner product on $\partial\Omega$ and $d\sigma$ denotes the induced surface measure on $\partial\Omega$.

  Suppose that $\chi\in C^\infty_0(\partial\Omega)$ is non-negative valued.  For $t<0$ sufficiently small, we have
  \begin{multline*}
    \left<\Hess_\nabla^b(Z,\bar Z)h,\chi\right>\geq\int_{\partial\Omega}(\bar Z h)(\psi(P,t))(\bar Z^*\bar\chi)(P) d\sigma_P\\
    -\int_{\partial\Omega}((\nabla_Z^b\bar Z) h)(\psi(P,t))\bar\chi(P) d\sigma_P-o(1),
  \end{multline*}
  where $o(1)$ represents an error term that vanishes uniformly as $t\rightarrow 0^-$.  Since $h\in C^2(U\backslash\partial\Omega)$, we may now integrate by parts in the first integral to obtain
  \begin{multline*}
    \left<\Hess_\nabla^b(Z,\bar Z)h,\chi\right>\geq\int_{\partial\Omega}Z|_P((\bar Z h)(\psi(P,t)))\bar\chi(P) d\sigma_P\\
    -\int_{\partial\Omega}((\nabla_Z^b\bar Z) h)(\psi(P,t))\bar\chi(P) d\sigma_P-o(1),
  \end{multline*}
  Since $\psi$ satisfies the variational equations \eqref{eq:varphi_variational_initial_condition} and \eqref{eq:varphi_variational_equation}, if we write
  \[
    Z=\sum_{j=1}^n b^j(z)\frac{\partial}{\partial z_j}
  \]
  in local holomorphic coordinates, then we must have
  \[
    \abs{(Z|_P)\psi^j(P,t)-b^j(\psi(P,t))}\leq \norm{X_r}_{C^1(U)}\norm{\psi}_{C^1(U)}O(t)
  \]
  and
  \[
    \abs{(Z|_P)\bar\psi^j(P,t)}\leq \norm{X_r}_{C^1(U)}\norm{\psi}_{C^1(U)}O(t)
  \]
  for all $1\leq j\leq n$, where $O(t)$ represents a term that is bounded by $Ct$ for some constant $C>0$ independent of $t$ and all $t$ sufficiently small.  Hence \eqref{eq:second_derivative_h_limit} gives us
  \begin{multline*}
    \left<\Hess_\nabla^b(Z,\bar Z)h,\chi\right>\geq\int_{\partial\Omega}((Z\bar Z h)(\psi(P,t)))\bar\chi(P) d\sigma_P\\
    -\int_{\partial\Omega}((\nabla_Z^b\bar Z) h)(\psi(P,t))\bar\chi(P) d\sigma_P-o(1).
  \end{multline*}
  Now we may substitute \eqref{eq:boundary_hessian_versus_hessian} to obtain
  \begin{multline*}
    \left<\Hess_\nabla^b(Z,\bar Z)h,\chi\right>\geq\int_{\partial\Omega}((\ddbar h(Z,\bar Z))(\psi(P,t)))\bar\chi(P) d\sigma_P\\
    -\int_{\partial\Omega}((\ddbar r(Z,\bar Z))X_rh)(\psi(P,t))\bar\chi(P) d\sigma_P-o(1).
  \end{multline*}
  We may assume that $0<-t<\delta$ so that we may substitute \eqref{eq:DF_index_necessary_condition} and obtain
  \begin{multline*}
    \left<\Hess_\nabla^b(Z,\bar Z)h,\chi\right>\geq\int_{\partial\Omega}((C-\epsilon)|Z|^2(\psi(P,t)))\bar\chi(P) d\sigma_P\\
    +\int_{\partial\Omega}\left(\left(i\beta_r(Z,\bar Z)+\frac{\eta}{1-\eta}\abs{\partial h(Z)-\alpha_r(Z)}^2\right)(\psi(P,t))\right)\bar\chi(P) d\sigma_P\\
    -(1+\epsilon)\int_{\partial\Omega}(-t)^{-1}(\ddbar r(Z,\bar Z))(P)\bar\chi(P) d\sigma_P\\
    -\int_{\partial\Omega}((\ddbar r(Z,\bar Z))X_r h)(\psi(P,t))\bar\chi(P) d\sigma_P-o(1).
  \end{multline*}
  Since all of the terms evaluated at $\psi(P,t)$ are continuous on $U$, we have shown that
  \begin{multline*}
    \Hess_\nabla^b(Z,\bar Z)h\geq(C-\epsilon)|Z|^2
    +i\beta_r(Z,\bar Z)+\frac{\eta}{1-\eta}\abs{\partial h(Z)-\alpha_r(Z)}^2\\
    -((1+\epsilon)(-t)^{-1}+X_rh)(\ddbar r(Z,\bar Z))-o(1)
  \end{multline*}
  in the distribution sense on $\partial\Omega$ for all $t<0$ sufficiently small.  Now, we fix $t$ sufficiently small so that
  \begin{multline*}
    \Hess_\nabla^b(Z,\bar Z)h\geq(C-2\epsilon)|Z|^2
    +i\beta_r(Z,\bar Z)+\frac{\eta}{1-\eta}\abs{\partial h(Z)-\alpha_r(Z)}^2\\
    -((1+\epsilon)(-t)^{-1}+X_rh)(\ddbar r(Z,\bar Z))
  \end{multline*}
  in the distribution sense on $\partial\Omega$ and regularize $h|_{\partial\Omega}$ and $((1+\epsilon)(-t)^{-1}+X_rh)|_{\partial\Omega}$ by convolution (as in Lemma \ref{lem:regularization} with $M$ replaced by $\partial\Omega$) to obtain new functions $h_0\in C^2(\partial\Omega)$ and $h_1\in C^1(\partial\Omega)$ satisfying
  \begin{multline*}
    \Hess_\nabla^b(Z,\bar Z)h_0\geq(C-3\epsilon)|Z|^2
    +i\beta_r(Z,\bar Z)+\frac{\eta}{1-\eta}\abs{\partial h_0(Z)-\alpha_r(Z)}^2\\
    -(\ddbar r(Z,\bar Z))h_1
  \end{multline*}
  on $\partial\Omega$.  Now we use Lemma \ref{lem:Whitney} to obtain a new function $h_\epsilon\in C^2(U)$ satisfying $h_\epsilon|_{\partial\Omega}=h_0$, $X_rh_\epsilon|_{\partial\Omega}=h_1$, and $X_r^2 h_\epsilon|_{\partial\Omega}=0$.  By \eqref{eq:boundary_hessian_versus_hessian}, if we relabel $h_\epsilon$ with $h$, we have \eqref{eq:DF_Index_necessary_condition_boundary}.

\end{proof}

\section{The Diederich Forn{\ae}ss Index: Sufficient Conditions}

\label{sec:DF_index_sufficient}

We now prove the converse to Lemma \ref{lem:DF_Index_necessary_condition}.
\begin{lem}
\label{lem:DF_Index_sufficient_condition}
  Let $M$ be a Hermitian manifold of complex dimension $n\geq 2$ and let $\Omega\subset M$ be a relatively compact pseudoconvex domain with $C^2$ boundary.  Let $U$ be a neighborhood of $\partial\Omega$ and let $r\in C^2(U)$ be an admissible defining function for $\Omega$.  Let $X_r$ be the unique $C^1$ section of $T(U)$ satisfying \eqref{eq:metric_compatibility}.  Let $\pi:U\rightarrow\partial\Omega$ be the unique $C^1$ projection satisfying $\pi(P)=P$ for all $P\in\partial\Omega$ and $X_r\pi\equiv 0$ on $U$.  Suppose that for some $0\leq\eta<1$ there exists a real-valued function $h\in C^1(U)\cap C^2(U\backslash\partial\Omega)$ satisfying \eqref{eq:second_derivative_h_limit} and constants $C>0$ and $0<\tilde\epsilon<1$ with the properties that for every $0<\epsilon<\tilde\epsilon$, there exists $\delta>0$ such that for every $C^1$ section $Z$ of $T^{1,0}(U)$ satisfying \eqref{eq:tangential_level_curves} and \eqref{eq:good_coordinate_ODE} and $z\in U\cap\Omega$ satisfying $0<-r(z)<\delta$, we have
  \begin{multline}
  \label{eq:DF_index_sufficient_condition}
    (1-\epsilon)(-r(z))^{-1}\ddbar r(Z,\bar Z)|_{\pi(z)}-i\beta_r(Z,\bar Z)|_z+\ddbar h(Z,\bar Z)|_z\\
    \geq\frac{\eta}{1-\eta}\abs{\partial h(Z)-\alpha_r(Z)}^2|_z+\left(C+2\epsilon\right)|Z|^2|_z.
  \end{multline}
  Then there exists a $C^2$ defining function $\rho$ for $\Omega$ such that \eqref{eq:DF_Index} (when $\eta>0$) or \eqref{eq:Strong_Oka_Property} (when $\eta=0$) holds on some interior neighborhood of $\partial\Omega$.  If $\Omega$ is Stein, then $\rho$ can be chosen so that \eqref{eq:DF_Index} (when $\eta>0$) or \eqref{eq:Strong_Oka_Property} (when $\eta=0$) holds on all of $\Omega$.
\end{lem}

\begin{proof}
  Let $0<\epsilon<\tilde\epsilon$ be given, and choose $\delta>0$ so that \eqref{eq:Levi_form_lower_bound} from Lemma \ref{lem:Levi_form_bounds} and \eqref{eq:DF_index_sufficient_condition} both hold when $0>r(z)>-\delta$.  Set $\rho=e^{-h}r$ on $U$.  Let $L_r$ be the unique $C^1$ section of $T^{1,0}(U)$ satisfying \eqref{eq:metric_compatibility_complex}.  Let $Z\in T^{1,0}(U)$ satisfy \eqref{eq:tangential_level_curves} and \eqref{eq:good_coordinate_ODE}.  For $x>0$ and $0\leq\eta<1$, set $\kappa_\eta(x)=\eta^{-1}(1-x^\eta)$ when $\eta>0$ and $\kappa_0(x)=-\log x$.  Using \eqref{eq:Lr_equals_one} with either \eqref{eq:rho_eta_hessian_identity} or \eqref{eq:rho_zero_hessian_identity}, we have
  \[
    (-\rho)^{-\eta}\ddbar(\kappa_\eta(-\rho))(L_r,\bar L_r)\geq(1-\eta)(-r)^{-2}-O((-r)^{-1})
  \]
  on $\Omega\cap U$, so we may assume that $\delta$ is sufficiently small so that $\ddbar(\kappa_\eta(-\rho))(L_r,\bar L_r)>C(-\rho)^\eta|L_r|^2$ whenever $0<-r(z)\leq\delta$.  For any fixed $z\in U$ at which $0<-r(z)<\delta$, the quadratic function $f:\mathbb{C}\rightarrow\mathbb{R}$ defined by
  \[
    f(w)=\left(\ddbar(\kappa_\eta(-\rho))+iC(-\rho)^\eta\omega\right)(Z+wL_r,\bar Z+\bar w\bar L_r)
  \]
  achieves its minimum at the critical point
  \[
    w_0=-\frac{\left(\ddbar(\kappa_\eta(-\rho))+iC(-\rho)^\eta\omega\right)(Z,\bar L_r)}{\left(\ddbar(\kappa_\eta(-\rho))+iC(-\rho)^\eta\omega\right)(L_r,\bar L_r)},
  \]
  and hence
  \begin{multline*}
    f(w)\geq f(w_0)=\\\left(\ddbar(\kappa_\eta(-\rho))+iC(-\rho)^\eta\omega\right)(Z,\bar Z)
    -\frac{\abs{\left(\ddbar(\kappa_\eta(-\rho))+iC(-\rho)^\eta\omega\right)(Z,\bar L_r)}^2}{\left(\ddbar(\kappa_\eta(-\rho))+iC(-\rho)^\eta\omega\right)(L_r,\bar L_r)}.
  \end{multline*}
  To prove either \eqref{eq:DF_Index} or \eqref{eq:Strong_Oka_Property} when $0<-r(z)<\delta$, it will suffice to show that
  \begin{equation}
  \label{eq:sufficient_condition_relation}
    \left(\ddbar(\kappa_\eta(-\rho))+iC(-\rho)^\eta\omega\right)(Z,\bar Z)
    \geq\frac{\abs{\left(\ddbar(\kappa_\eta(-\rho))+iC(-\rho)^\eta\omega\right)(Z,\bar L_r)}^2}{\left(\ddbar(\kappa_\eta(-\rho))+iC(-\rho)^\eta\omega\right)(L_r,\bar L_r)}
  \end{equation}
  for all $Z\in T^{1,0}(U)$ satisfying \eqref{eq:tangential_level_curves} and \eqref{eq:good_coordinate_ODE}.  Note that either \eqref{eq:rho_eta_hessian_identity} or \eqref{eq:rho_zero_hessian_identity} with \eqref{eq:Lr_equals_one} and \eqref{eq:tangential_level_curves} imply that
  \[
    \abs{(-\rho)^{-\eta}\ddbar(\kappa_\eta(-\rho))(Z,\bar L_r)}^2\leq(-r)^{-2}\abs{\alpha_r(Z)-\eta\partial h(Z)}^2+O((-r)^{-1}),
  \]
  so we have
  \begin{multline}
  \label{eq:sufficient_condition_right_hand_side}
    \frac{\abs{\left(\ddbar(\kappa_\eta(-\rho))+iC(-\rho)^\eta\omega\right)(Z,\bar L_r)}^2}{\left(\ddbar(\kappa_\eta(-\rho))+iC(-\rho)^\eta\omega\right)(L_r,\bar L_r)}\leq\\
    \frac{1}{1-\eta}(-\rho)^{\eta}\abs{\alpha_r(Z)-\eta\partial h(Z)}^2+O((-r)^{\eta+1}).
  \end{multline}

  A further application of either \eqref{eq:rho_eta_hessian_identity} or \eqref{eq:rho_zero_hessian_identity} gives us
  \begin{multline*}
    (-\rho)^{-\eta}\ddbar(\kappa_\eta(-\rho))(Z,\bar Z)\geq\\
    (-r)^{-1}\ddbar r(Z,\bar Z)-\eta|\partial h(Z)|^2+\ddbar h(Z,\bar Z)-O(-r).
  \end{multline*}
  If we substitute \eqref{eq:Levi_form_lower_bound} and \eqref{eq:DF_index_sufficient_condition}, we see that
  \begin{multline*}
    (-\rho)^{-\eta}\ddbar(\kappa_\eta(-\rho))(Z,\bar Z)\geq\\
    |\alpha_r(Z)|^2-\eta|\partial h(Z)|^2+\frac{\eta}{1-\eta}\abs{\partial h(Z)-\alpha_r(Z)}^2+(C+\epsilon)|Z|^2-O(-r).
  \end{multline*}
  whenever $z\in U$ satisfies $0<-r(z)<\delta$.  This is equivalent to
  \[
    (-\rho)^{-\eta}\ddbar(\kappa_\eta(-\rho))(Z,\bar Z)\geq
    \frac{1}{1-\eta}\abs{\eta\partial h(Z)-\alpha_r(Z)}^2+(C+\epsilon)|Z|^2-O(-r),
  \]
  so
  \begin{multline*}
    \left(\ddbar(\kappa_\eta(-\rho))+iC(-\rho)^\eta\omega\right)(Z,\bar Z)\geq\\
    \frac{1}{1-\eta}(-\rho)^{\eta}\abs{\eta\partial h(Z)-\alpha_r(Z)}^2+\epsilon(-\rho)^{\eta}|Z|^2-O((-r)^{\eta+1}).
  \end{multline*}
  Combined with \eqref{eq:sufficient_condition_right_hand_side}, we see that we may further shrink $\delta$ so that \eqref{eq:sufficient_condition_relation}, and hence either \eqref{eq:DF_Index} or \eqref{eq:Strong_Oka_Property} holds near all $z\in U$ satisfying $0<-r(z)<\delta$.

  If $\Omega$ is Stein, it admits a strictly plurisubharmonic exhaustion function $\phi$.  We may assume that $\phi$ is unbounded, so that $\lim_{z\rightarrow\partial\Omega}\phi(z)=\infty$.  By subtracting a sufficiently large constant, we may assume that if $z\in\Omega$ satisfies $\phi(z)=0$, then $0<-r(z)<\delta$ and
  \[
    \inf_{\{w\in\Omega:\phi(w)=0\}}\rho(w)>\sup_{\{w\in U:r(w)=-\delta\}}\rho(w).
  \]
  When $\eta>0$, we may set
  \[
    A=\sup_{\{w\in\Omega:\phi(w)=0\}}(-\rho(w))^\eta
  \]
  and
  \[
    B=\left(\inf_{\{w\in U:r(w)=-\delta\}}(-\rho(w))^\eta-A\right)\left(\sup_{\{w\in U:r(w)=-\delta\}}(-\phi(w))\right)^{-1}
  \]
  and know that both constants are strictly positive.  Set
  \[
    \lambda(z)=\begin{cases}
      -(-\rho(z))^\eta&\phi(z)\geq 0,\\
      \max\{-(-\rho(z))^\eta,B\phi(z)-A\}&\phi(z)<0\text{ and }-r(z)<\delta,\\
      B\phi(z)-A&-r(z)\geq\delta\text{ or }z\notin U.
    \end{cases}
  \]
  This is continuous on $\Omega$, $C^2$ except on a compact subset of $U\cap\Omega$, and strictly plurisubharmonic on $\Omega$.  After a standard regularization argument, we may replace $\lambda$ with a $C^2$ function $\tilde\lambda$ and set $\tilde\rho=-(-\tilde\lambda)^{1/\eta}$.  When $\eta=0$, the argument is similar apart from the natural modifications.
\end{proof}

Now we are ready to prove the converse to Proposition \ref{prop:DF_Index_necessary_condition_boundary}.
\begin{prop}
\label{prop:DF_Index_sufficient_condition_boundary}
  Let $M$ be a Hermitian manifold of complex dimension $n\geq 2$ and let $\Omega\subset M$ be a relatively compact pseudoconvex domain with $C^2$ boundary.  Let $U$ be a neighborhood of $\partial\Omega$ and let $r\in C^2(U)$ be an admissible defining function for $\Omega$.  Suppose that for some $0\leq\eta<1$ there exists a real-valued function $h\in C^2(U)$ and a constant $\tilde C>0$ with the property that
  \begin{equation}
  \label{eq:DF_index_sufficient_condition_boundary}
    -i\beta_r(Z,\bar Z)+\ddbar h(Z,\bar Z)\\
    >\frac{\eta}{1-\eta}\abs{\partial h(Z)-\alpha_r(Z)}^2+\tilde C|Z|^2
  \end{equation}
  for every $P\in\partial\Omega$ and $Z\in\mathcal{N}_P(\partial\Omega)$.  Then for every $0<C<\tilde C$ there exists a $C^2$ defining function $\rho$ for $\Omega$ such that either \eqref{eq:DF_Index} (for $\eta>0$) or \eqref{eq:Strong_Oka_Property} (for $\eta=0$) holds on some interior neighborhood of $\partial\Omega$.  If $\Omega$ is Stein, then $\rho$ can be chosen so that \eqref{eq:DF_Index} (when $\eta>0$) or \eqref{eq:Strong_Oka_Property} (when $\eta=0$) holds on all of $\Omega$.
\end{prop}

\begin{proof}
  Let $X_r$ be the unique $C^1$ section of $T(U)$ satisfying \eqref{eq:metric_compatibility}, and let $\pi:U\rightarrow\partial\Omega$ be the unique $C^1$ projection satisfying $\pi(P)=P$ for all $P\in\partial\Omega$ and $X_r\pi\equiv 0$ on $U$.

  Fix $0<C<\tilde C$, and set $\tilde\epsilon=\min\set{1,\frac{\tilde C-C}{2}}$.  Suppose there exists $0<\epsilon<\tilde\epsilon$ such that for every $j\in\mathbb{N}$ there exists a $C^1$ section $Z_j$ of $T^{1,0}(U)$ satisfying \eqref{eq:tangential_level_curves} and \eqref{eq:good_coordinate_ODE} and $z_j\in U\cap\Omega$ satisfying $0<-r(z_j)<\frac{1}{j}$ such that \eqref{eq:DF_index_sufficient_condition} fails for $Z_j$ at $z_j$, i.e.,
  \begin{multline*}
    (1-\epsilon)(-r(z_j))^{-1}\ddbar r(Z_j,\bar Z_j)|_{\pi(z_j)}-i\beta_r(Z_j,\bar Z_j)|_{z_j}+\ddbar h(Z_j,\bar Z_j)|_{z_j}\\
    <\frac{\eta}{1-\eta}\abs{\partial h(Z_j)-\alpha_r(Z_j)}^2|_{z_j}+\left(C+2\epsilon\right)|Z_j|^2|_{z_j}.
  \end{multline*}
  As in the proof of Lemma \ref{lem:Levi_form_bounds}, $|Z_j||_{z_j}=|Z_j||_{\pi(z_j)}$.  If $|Z_j||_{z_j}=|Z_j||_{\pi(z_j)}=0$, then \eqref{eq:DF_index_sufficient_condition} holds trivially, so we may normalize and assume that $|Z_j||_{z_j}=|Z_j||_{\pi(z_j)}=1$ for all $j\in\mathbb{N}$.  After restricting to a subsequence, we may assume that $z_j\rightarrow P$ for some $P\in\partial\Omega$ and $Z_j|_{z_j}\rightarrow Z$ for some $Z\in T_P^{1,0}(\partial\Omega)$.  We must have $Z_j|_{\pi(z_j)}\rightarrow Z$ as well.  Since $(1-\epsilon)(-r(z_j))^{-1}\ddbar r(Z_j,\bar Z_j)|_{\pi(z_j)}\geq 0$ for all $j\in\mathbb{N}$ and this term admits a uniform upper bound by assumption, we must have $\ddbar r(Z,\bar Z)=0$ at $P$.  Since $\partial\Omega$ is pseudoconvex, this means that $Z\in\mathcal{N}_P(\partial\Omega)$.  Furthermore, we have
  \[
    -i\beta_r(Z_j,\bar Z_j)|_{z_j}+\ddbar h(Z_j,\bar Z_j)|_{z_j}\\
    <\frac{\eta}{1-\eta}\abs{\partial h(Z_j)-\alpha_r(Z_j)}^2|_{z_j}+\left(C+2\epsilon\right)|Z_j|^2|_{z_j}
  \]
  for all $j\in\mathbb{N}$, so we have
  \[
    -i\beta_r(Z,\bar Z)|_P+\ddbar h(Z,\bar Z)|_P
    \leq\frac{\eta}{1-\eta}\abs{\partial h(Z)-\alpha_r(Z)}^2|_P+\left(C+2\epsilon\right)|Z|^2|_P.
  \]
  However, $C+2\epsilon<C+2\tilde\epsilon\leq\tilde C$, so we have contradicted \eqref{eq:DF_index_sufficient_condition_boundary}.  Hence, \eqref{eq:DF_index_sufficient_condition} must hold, and the conclusion follows from Lemma \ref{lem:DF_Index_sufficient_condition}.
\end{proof}

\section{Proofs of the Main Theorems}

\label{sec:proofs}

\begin{proof}[Proof of Theorem \ref{thm:boundary_equivalence}]
  We see that (1) implies (2) by Proposition \ref{prop:DF_Index_necessary_condition_boundary} and regularization of $h$.  It is trivial that (2) implies (3).

  Suppose that for some $0\leq\eta<1$, $C>0$, and $h\in C^2(\partial\Omega)$, \eqref{eq:DF_Index_boundary} holds for all $P\in\partial\Omega$ and $Z\in\mathcal{N}_P(\partial\Omega)$.  For any $0<\tilde C<C$, there exists $\eta<\tilde\eta<1$ such that \eqref{eq:DF_Index_boundary} also holds if we replace $C$ with $\tilde C$ and $\eta$ with $\tilde\eta$.  Hence Proposition \ref{prop:DF_Index_sufficient_condition_boundary} implies that $\tilde\eta$ is also a strong Diederich-Forn{\ae}ss exponent, and so $DF(\Omega)>\eta$.  This prove that (3) implies (1).

  From the proof of Lemma \ref{lem:DF_Index_sufficient_condition}, it can be seen that $\rho=r e^{-h}$ satisfies either \eqref{eq:DF_Index} (when $\eta>0$) or \eqref{eq:Strong_Oka_Property} (when $\eta=0$) on some interior neighborhood of $\partial\Omega$.
\end{proof}

\begin{proof}[Proof of Theorem \ref{thm:geometric_equivalence}]
  Suppose $0\leq\eta<DF(\Omega)$.  Fix a neighborhood $U$ of $\partial\Omega$ and an admissible defining function $r\in C^2(U)$ for $\Omega$ and let $h\in C^\infty(U)$ be the function given by (2) in Theorem \ref{thm:boundary_equivalence}.  Let $\left<\cdot,\cdot\right>_0$ be an arbitrary Hermitian metric on $M$, and let $L_r^0$ be the unique $C^1$ section of $T^{1,0}(U)$ satisfying \eqref{eq:metric_compatibility_complex} with respect to $\left<\cdot,\cdot\right>_0$.  Define a new Hermitian metric by $\left<Z,W\right>_1=2e^{-2h}|\partial r|_0^{2}\left<Z,W\right>_0$ for all $Z,W\in T^{1,0}(U)$.  Since $r$ is admissible, this metric has $C^2$ coefficients with respect to any local holomorphic coordinates.  Using \eqref{eq:dual_norm_identity_complex}, $|L_r^0|_1=\sqrt{2}e^{-h}$.  Since this is a conformal change of metric, $L_r^1=L_r^0$ is the unique $C^1$ section satisfying \eqref{eq:metric_compatibility_complex} with respect to $\left<\cdot,\cdot\right>_1$.  By \eqref{eq:dual_norm_identity_complex}, $|\partial r|_1=\frac{1}{\sqrt{2}}e^h$, so $|dr|_1=e^h$.  Note that we can take $\nu_{\mathbb{R}}=|X_r^1|_1^{-1}X_r^1$ and $\nu_{\mathbb{C}}=|L_r^1|_1^{-1}L_r^1$, where $X_r^1$ is the unique $C^1$ section of $T(U)$ satisfying \eqref{eq:metric_compatibility} with respect to $\left<\cdot,\cdot\right>_1$.
  For any $Z\in\mathcal{N}_P(\partial\Omega)$, \eqref{eq:alpha_r_geometric} implies
  \[
    \alpha_r(Z)=Zh-i\left<\sff_{\nabla^1}(Z,J\nu_{\mathbb{R}}),\nu_{\mathbb{R}}\right>_1,
  \]
  and \eqref{eq:beta_r_geometric} implies
  \[
    \beta_r(Z,\bar Z)|_P=-i\ddbar h(Z,\bar Z)
    +i\sum_{j=1}^{n-1}\abs{\sff_{\nabla^1}(Z,W_j)}_1^2+\frac{i}{2}\left<R_{\nabla^1}(Z,\bar Z)\nu_{\mathbb{C}},\nu_{\mathbb{C}}\right>_1.
  \]
  Substituting these in \eqref{eq:DF_Index_boundary} gives us \eqref{eq:DF_Index_boundary_geometry}.  By decreasing the value of $C$, \eqref{eq:DF_Index_boundary_geometry} is stable under $C^2$ perturbations of the metric, so we may regularize and obtain a smooth metric also satisfying \eqref{eq:DF_Index_boundary_geometry}.  Hence (1) implies (2).

  To see that (2) and (3) are equivalent, it will suffice to show that \eqref{eq:DF_Index_boundary_geometry} and \eqref{eq:DF_Index_boundary_vector_field} are equivalent.  Fix $P\in\partial\Omega$ and $Z\in\mathcal{N}_P(\partial\Omega)$.  Given an orthonormal basis $\{W_j\}_{j=1}^{n-1}$ for $T^{1,0}_P(\partial\Omega)$, then since $\nabla_{\bar Z}\nu_{\mathbb{C}}-\left<\nabla_{\bar Z}\nu_{\mathbb{C}},\nu_{\mathbb{C}}\right>\nu_{\mathbb{C}}\in T^{1,0}_P(\partial\Omega)$, we have
  \[
    \abs{\nabla_{\bar Z}\nu_{\mathbb{C}}-\left<\nabla_{\bar Z}\nu_{\mathbb{C}},\nu_{\mathbb{C}}\right>\nu_{\mathbb{C}}}^2=\sum_{j=1}^{n-1}\abs{\left<\nabla_{\bar Z}\nu_{\mathbb{C}},W_j\right>}^2.
  \]
  Since $\nu_{\mathbb{R}}=\frac{\nu_{\mathbb{C}}+\bar\nu_{\mathbb{C}}}{\sqrt{2}}$, we have
  \begin{equation}
  \label{eq:dbar_nu_tangential}
    \abs{\nabla_{\bar Z}\nu_{\mathbb{C}}-\left<\nabla_{\bar Z}\nu_{\mathbb{C}},\nu_{\mathbb{C}}\right>\nu_{\mathbb{C}}}^2=2\sum_{j=1}^{n-1}\abs{\left<\nabla_{\bar Z}\nu_{\mathbb{R}},W_j\right>}^2=2\sum_{j=1}^{n-1}\abs{\sff_\nabla(\bar Z,\bar W_j)}^2.
  \end{equation}
  We also have
  \[
    \abs{\sff_\nabla(\bar Z,J\nu_{\mathbb{R}})}^2=\abs{\left<\nabla_{\bar Z}\nu_{\mathbb{R}},J\nu_{\mathbb{R}}\right>}^2.
  \]
  Expanding in terms of $\nu_{\mathbb{C}}$ again gives us
  \[
    \abs{\sff_\nabla(\bar Z,J\nu_{\mathbb{R}})}^2=\frac{1}{4}\abs{\left<\nabla_{\bar Z}\bar \nu_{\mathbb{C}},\bar \nu_{\mathbb{C}}\right>-\left<\nabla_{\bar Z}\nu_{\mathbb{C}},\nu_{\mathbb{C}}\right>}^2.
  \]
  Since $|\nu_{\mathbb{C}}|^2$ is constant, we may use the metric compatibility of $\nabla$ to obtain
  \begin{equation}
  \label{eq:dbar_nu_normal}
    \abs{\sff_\nabla(\bar Z,J\nu_{\mathbb{R}})}^2=\abs{\left<\nabla_{\bar Z}\nu_{\mathbb{C}},\nu_{\mathbb{C}}\right>}^2.
  \end{equation}
  From \eqref{eq:dbar_nu_tangential} and \eqref{eq:dbar_nu_normal}, we see that \eqref{eq:DF_Index_boundary_geometry} and \eqref{eq:DF_Index_boundary_vector_field} are equivalent.

  Now suppose that we already have a Hermitian metric satisfying \eqref{eq:DF_Index_boundary_geometry} for some $0\leq\eta<1$, then let $U$ be a neighborhood of $\partial\Omega$ on which there exists a unique $r\in C^2(U)$ such that $|dr|\equiv 1$ on $U$, i.e., $r$ is the signed distance function, and set $h\equiv 0$ on $U$.  Then \eqref{eq:DF_Index_boundary} follows from \eqref{eq:DF_Index_boundary_geometry} with \eqref{eq:alpha_r_geometric} and \eqref{eq:beta_r_geometric}.  Hence, $\eta<DF(\Omega)$ by Theorem \ref{thm:boundary_equivalence}, so (2) implies (1).  As noted in Theorem \ref{thm:boundary_equivalence}, $\rho=r e^0=r$ satisfies either \eqref{eq:DF_Index} (when $\eta>0$) or \eqref{eq:Strong_Oka_Property} (when $\eta=0$), so (2) implies (4).

  If we consider the proof of Lemma \ref{lem:DF_Index_necessary_condition} and Proposition \ref{prop:DF_Index_necessary_condition_boundary}, we see that if $\rho=r$, then we may take $h=0$, which means that no regularization of $h$ is required.  Hence, Proposition \ref{prop:DF_Index_necessary_condition_boundary} with $\rho=r$ equal to the signed distance function and $h=0$ proves that (4) implies (2) via \eqref{eq:alpha_r_geometric} and \eqref{eq:beta_r_geometric}.
\end{proof}

\begin{proof}[Proof of Corollary \ref{cor:complex_submanifold_Kahler}]
  Suppose that $S\subset\partial\Omega$ is a compact complex submanifold of positive dimension.  Let $U$ be a neighborhood of $\partial\Omega$ and let $r\in C^2(U)$ be an admissible defining function for $\Omega$.  Let $0<\eta<DF(\Omega)$. Fix $C>0$ and $h\in C^\infty(U)$ satisfying \eqref{eq:DF_Index_boundary} for all $Z\in T^{1,0}(U)$.  By Proposition \ref{prop:alpha_closed}, $\alpha_r|_S$ is $d_S$-closed, so the Hodge decomposition implies that $\alpha_r|_S=d_S u+\tilde\alpha_r$, where $u\in C^1(S)$ and $\tilde\alpha_r$ lies in the kernel of $d_S$ and $d^*_S$.  Here, $d^*_S$ denotes the Hilbert space adjoint of $d_S$ with respect to $L^2(S)$.  Since $\tilde\alpha_r$ is harmonic, it must also be smooth.  By Proposition \ref{prop:beta}, $\beta_r|_S=-i\partial_S\dbar_S u-\frac{i}{2}(\partial_S-\dbar_S)\tilde\alpha_r$ in the sense of currents.  Since $\tilde\alpha_r$ is smooth and $\beta_r$ is continuous, this also means that $i\partial_S\dbar_S u$ can be identified with a continuous $(1,1)$-form.  By Theorem \ref{thm:boundary_equivalence}, the $(1,1)$-current on $S$ defined by $\theta(Z,\bar W)=-\frac{i}{2}(\partial_S-\dbar_S)\tilde\alpha_r(Z,\bar W)+i\partial_S\dbar_S(h-u)(Z,\bar W)$ for all $Z,W\in T^{1,0}(S)$ is necessarily positive definite on $S$.  If we regularize $h-u$ by convolution, we obtain a smooth, positive-definite, real-valued $(1,1)$-form on $S$.  Hence, $S$ is K\"ahler.

  Henceforth, we endow $S$ with a K\"ahler metric and extend this metric to a Hermitian metric on $M$.  Note that this may require a new choice of $h\in C^\infty(U)$ and $C>0$.  Let $\tilde\alpha_r$ and $u$ be as before, and set $\tilde\beta_r=-\frac{i}{2}(\partial_S-\dbar_S)\tilde\alpha_r$ on $S$, so that $\beta_r|_S=i\partial_S\dbar_S u+\tilde\beta_r$.  We decompose $\tilde\alpha_r=\gamma_r+\bar\gamma_r$, where $\gamma_r\in\Lambda^{1,0}(U)$ is a smooth $(1,0)$-form.  Since $d_S\tilde\alpha_r=0$, we have $\partial_S\gamma_r=0$ and $\dbar_S\gamma_r+\partial_S\bar\gamma_r=0$.  This means that $\tilde\beta_r=d_S\left(\frac{i}{2}(\gamma_r-\bar\gamma_r)\right)$, so $\tilde\beta_r$ is $d_S$-exact.  On the other hand, $d^*_S\tilde\alpha_r=0$ implies that $\alpha_r$ is harmonic with respect to the $d_S$-Laplacian $\Delta_{d_S}=d_S d_S^*+d_S^* d_S$.  Since we are working with a K\"ahler metric, we also have $\Box_S\alpha_r=0$, where $\Box_S=\dbar_S\dbar^*_S+\dbar^*_S\dbar_S$ denotes the $\dbar_S$-Laplacian.  Since $\Box_S$ preserves the degrees of differential forms, this means that we have $\Box_S\gamma_r=0$ and $\Box_S\bar\gamma_r=0$.  We can write $\tilde\beta_r=i\dbar_S\gamma_r$, so $\dbar^*_S\tilde\beta_r=-i\dbar_S\dbar_S^*\gamma_r$.  Since the range of $\dbar^*_S$ is orthogonal to the range of $\dbar_S$, we must have $\dbar^*_S\tilde\beta_r=0$.  However, these means that $\tilde\beta_r\in\ker\dbar^*_S$ and $\tilde\beta_r\in\range\dbar_S$, but these spaces are also orthogonal, so we must have $\tilde\beta_r=0$.  This means that $\beta_r=i\partial_S\dbar_S u$, and so Theorem \ref{thm:boundary_equivalence} implies that $h-u$ is strictly plurisubharmonic on $S$.  However, no such functions exist on a compact complex manifold of dimension one or greater (even with the relaxed regularity assumption), so $S$ must be dimension zero.
\end{proof}

\appendix

\section{Analysis on \texorpdfstring{$C^2$}{C2} Manifolds}

\label{sec:analysis_on_manifolds}

The following regularity result is a straightforward application of the method of characteristics and the basic theory of ordinary differential equations (see, for example, Section 17.6 of \cite{HSD13}).  We outline the proof here primarily to illustrate that it suffices for $X$ to be a $C^1$ section.
\begin{lem}
\label{lem:characteristic_regularity}
  Let $M$ be a $C^2$ manifold of real dimension $n\geq 2$, let $\Omega\subset M$ be a relatively compact domain with $C^2$ boundary, and let $X$ be a $C^1$ section of $T(M)$ that is transverse to $\partial\Omega$.  Then there exists a neighborhood $U$ of $\partial\Omega$ so that for every $N\in\mathbb{N}$ and $\{f_j\}_{j=1}^N\subset C(U)$ with the properties
  \begin{enumerate}
    \item $f_j|_{\partial\Omega}\in C^1(\partial\Omega)$ for all $1\leq j\leq N$,
    \item $Xf_j$ exists on $U$ for all $1\leq j\leq N$, and
    \item $Xf_j=\sum_{k=1}^N A_j^k f_k+B_j$ for all $1\leq j\leq N$, where $\{A_j^k\}_{1\leq j,k\leq N}\subset C^1(U)$ and $\{B_j\}_{1\leq j\leq N}\subset C^1(U)$,
  \end{enumerate}
  we have
  \begin{enumerate}
    \item $f_j\in C^1(U)$ for all $1\leq j\leq N$,
    \item $X(Yf_j)$ exists on $U$ for every $C^1$ section $Y$ of $T(U)$ and $1\leq j\leq N$, and
    \item \begin{equation}\label{eq:variational_equation}X(Yf_j)=[X,Y]f_j+Y(Xf_j)\end{equation}
        on $U$ for all $C^1$ sections $Y$ of $T(U)$ and $1\leq j\leq N$.
  \end{enumerate}
\end{lem}

\begin{proof}
  For $P\in\partial\Omega$, let $U_P$ be a neighborhood of $P$ on which there exist local $C^2$ coordinates $\{x_j\}_{j=1}^n$ such that $\frac{\partial}{\partial x_j}\in T_P(\partial\Omega)$ for all $1\leq j\leq n-1$.  If we write $X=\sum_{j=1}^n a^j(x)\frac{\partial}{\partial x_j}$ for functions $a^j\subset C^1(U_P)$, then the basic theory of ordinary differential equations guarantees that for every point $Q\in\partial\Omega\cap U_P$ there exists some $\delta_Q>0$ and a function $\psi(Q,\cdot):(-\delta_Q,\delta_Q)\rightarrow U$ such that $\psi(Q,0)=Q$ and $\frac{\partial}{\partial t}\psi^j(Q,t)=a^j(\psi(Q,t))$ whenever $|t|<\delta_Q$ and $1\leq j\leq n$.  Furthermore, for all $1\leq j\leq n$ and $Y\in T_Q(\partial\Omega)$, $Y\psi^j(Q,t)$ satisfies the variational equation
  \begin{equation}
  \label{eq:varphi_variational_initial_condition}
    Y\psi^j(Q,0)=dx_j(Y)|_Q
  \end{equation}
  and
  \begin{equation}
  \label{eq:varphi_variational_equation}
    \frac{\partial}{\partial t}\left(Y\psi^j(Q,t)\right)=\sum_{\ell=1}^n \frac{\partial a^j}{\partial x_k}(\psi(Q,t))Y\psi^k(Q,t).
  \end{equation}

  We may shrink $U_P$ so that $\delta_Q\leq\delta$ for some fixed $\delta>0$ and any $Q\in\partial\Omega\cap U_P$.  We further assume that $U_P$ is sufficiently small so that we may let $\{X_j\}_{j=1}^{n-1}$ be linearly independent $C^1$ sections of $T(\partial\Omega\cap U_P)$ such that $X_j|_P=\frac{\partial}{\partial x_j}$.  For each $Q\in\partial\Omega\cap U_P$ and $t\in(-\delta,\delta)$, define a matrix $M(Q,t)=\left(M^j_k(Q,t)\right)_{1\leq j,k\leq n}$ by $M^j_k(Q,t)=X_k\varphi^j(Q,t)$ for $1\leq j\leq n$ and $1\leq k\leq n-1$ and $M^j_n(Q,t)=a^j(\varphi(Q,t))$ for $1\leq j\leq n$.  By \eqref{eq:varphi_variational_initial_condition}, $M^j_k(P,0)=I_{jk}$ for all $1\leq j\leq n$ and $1\leq k\leq n-1$.  Since $X$ is transverse to $\partial\Omega$, $M^n_n(P,0)=a^n(P)\neq 0$, and hence $M(P,0)$ is a non-singular matrix.  The Implicit Function Theorem guarantees that we may further shrink $U_P$ so that there exists a $C^1$ map $\pi:U_P\rightarrow\partial\Omega$ and a $C^1$ function $r:U_P\rightarrow\mathbb{R}$ such that $\psi(\pi(x),r(x))=x$ on $U_P$.  For $Q\in\partial\Omega\cap U_P$, $\psi(Q,0)=Q$, so the uniqueness of $\pi$ and $r$ guarantees that $\pi(Q)=Q$ and $r(Q)=0$ for all $Q\in\partial\Omega\cap U_P$.  For $Q\in\partial\Omega\cap U_P$ and $t\in\mathbb{R}$ sufficiently small, we must have $\pi(\psi(Q,t))=Q$ and $r(\psi(Q,t))=t$.  Differentiating each identity by $t$, we find that $X\pi\equiv 0$ and $Xr\equiv 1$ on $U_P$.

  Since $\frac{\partial}{\partial t}f_m(\psi(Q,t))=Xf_m(\psi(Q,t))$ for each $1\leq m\leq N$, the basic theory of ordinary differential equations implies that $f_m(\psi(Q,t))$ is also $C^1$ in both $Q$ and $t$.  Since $f_m(x)$ is the composition of the $C^1$ function $f_m(\psi(Q,t))$ with the $C^1$ maps $Q=\pi(x)$ and $t=r(x)$, $f_m$ must also lie in $C^1(U_P)$.  Furthermore, derivatives of $f_m(\psi(Q,t))$ with respect to $X_k$ for $1\leq k\leq n-1$ must satisfy the variational equation
  \[
    \frac{\partial}{\partial t}(X_k(f_m(\psi(Q,t))))=X_k(Xf_m(\psi(Q,t))).
  \]
  Expanding and simplifying with \eqref{eq:varphi_variational_equation}, we can obtain \eqref{eq:variational_equation}.  We omit the details.

  If we write $U=\bigcup_{p\in P}U_P$, we are done.
\end{proof}

The following is a simple consequence of the Whitney Extension Theorem.
\begin{lem}
\label{lem:Whitney}
  Let $M$ be a $C^2$ manifold of real dimension $n\geq 2$.  Let $\Omega\subset M$ be a relatively compact domain with $C^2$ boundary.  Let $X$ be a $C^1$ section of $T(M)$ that is transverse to $\partial\Omega$.  Given $f_0\in C^2(\partial\Omega)$, $f_1\in C^1(\partial\Omega)$, and $f_2\in C(\partial\Omega)$, then there exists $h\in C^2(M)$ such that $h|_{\partial\Omega}=f_0$, $Xh|_{\partial\Omega}=f_1$, and $X^2h|_{\partial\Omega}=f_2$.
\end{lem}

\begin{proof}
It suffices to note that the hypotheses provide enough data to uniquely determine the $2$-jet of $h$ at any point $P\in\partial\Omega$, and it is not difficult to check that these $2$-jets must satisfy the hypotheses of the Whitney Extension Theorem.  We omit the details.

\end{proof}

Finally, we analyze regularization by convolution.
\begin{lem}
\label{lem:regularization}
  Let $M$ be a compact $C^2$ manifold of real dimension $n\geq 2$ equipped with a Riemannian metric.  For every $\epsilon>0$, there exists an integral kernel $K_\epsilon\in C^2(M\times M)$ taking non-negative values such that
  \begin{enumerate}
    \item $K_\epsilon(x,y)=0$ whenever $\dist(x,y)\geq\epsilon$,
    \item $\int_M K_\epsilon(x,y)dV_y=1$ for all $x\in M$,
    \item $\lim_{\epsilon\rightarrow 0^+}\int_M f(y)K_\epsilon(x,y)dV_y=f(x)$ for all $f\in C(M)$ and $x\in M$,
    \item $\lim_{\epsilon\rightarrow 0^+}X\int_M f(y)K_\epsilon(x,y)dV_y=Xf(x)$ for all $f\in C^1(M)$, $x\in M$, and continuous vector fields $X\in T(M)$, and
    \item if $\nabla$ is any affine connection on $T(M)$, then
  \begin{multline}
  \label{eq:second_derivative_convergence}
    \Hess_\nabla(X,Y)\int_M f(y)K_\epsilon(x,y)dV_y\\-\int_{\mathbb{R}^n}(Xf)(y)Y^*_y\left(K_\epsilon(x,y)\right)dV_y+\int_{\mathbb{R}^n}((\nabla_Y X)f)(y)K_\epsilon(x,y)dV_y\rightarrow 0
  \end{multline}
  as $\epsilon\rightarrow 0^+$ for all $f\in C^1(M)$, $x\in M$, and $C^1$ sections $X$ and $Y$ of $T(M)$.
  \end{enumerate}
\end{lem}

\begin{proof}
  Note that (1) and (2) always imply (3).  With this in mind, it is straightforward to check that we can construct $K_\epsilon(x,y)$ locally and use a partition of unity (with respect to the variable $x$) to construct a global kernel.

  Henceforth, we assume that $f$ is supported in some local coordinate chart.  On this coordinate chart, there must exist some constant $C>0$ such that $|x-y|\leq C\dist(x,y)$.  Define $\chi\in C^2(\mathbb{R})$ by $\chi(t)=(1-t^2)^3$ when $|t|<1$ and $\chi(t)=0$ when $|t|\geq 1$.  Set $\tilde K_\epsilon(x,y)=\chi(\epsilon^{-1}C^{-1}|x-y|)/\int_{\mathbb{R}^n}\chi(\epsilon^{-1}C^{-1}|z|)dz_1\ldots dz_n$.  Then $\tilde K_\epsilon(x,y)$ satisfies (1) and
  \begin{equation}
  \label{eq:tilde_K_integral}
    \int_{\mathbb{R}^n}\tilde K_\epsilon(x,y)dy_1\ldots dy_n=1\text{ for all }x\in\mathbb{R}^n.
  \end{equation}
  Furthermore, we have
  \begin{equation}
  \label{eq:tilde_K_derivative_relation}
    \frac{\partial}{\partial x_j}\tilde K_\epsilon(x,y)+\frac{\partial}{\partial y_j}\tilde K_\epsilon(x,y)\equiv 0
  \end{equation}
  for all $1\leq j\leq n$.  Differentiating \eqref{eq:tilde_K_integral} gives us
  \begin{equation}
  \label{eq:tilde_K_derivative_identity}
    \int_{\mathbb{R}^n}\frac{\partial}{\partial x_j}\tilde K_\epsilon(x,y)dy_1\ldots dy_n=0
  \end{equation}
  for all $1\leq j\leq n$.  We may directly estimate
  \begin{equation}
  \label{eq:tilde_K_derivative_estimate}
    \int_{\mathbb{R}^n}\abs{\frac{\partial}{\partial x_j}\tilde K_\epsilon(x,y)dy_1\ldots dy_n}\leq O\left(\frac{1}{\epsilon}\right)
  \end{equation}
  for $\epsilon>0$ sufficiently small.

  Note that we may write $dV_y=h(y)dy_1\ldots dy_n$ for some positive-valued $C^1$ function $h$.  Set $h_\epsilon(x)=\int_{\mathbb{R}^n}\tilde K_\epsilon(x,y)dV_y$.  Since $K_\epsilon(x,y)$ is $C^2$ with respect to $x$, $h_\epsilon$ is also $C^2$.  Moreover, $\lim_{\epsilon\rightarrow 0^+}h_\epsilon(x)=h(x)$.  If we make the usual change of coordinates for convolutions, we have
  \[
    h_\epsilon(x)=\epsilon^n\int_{\mathbb{R}^n}\tilde K_\epsilon(x,x+\epsilon z)h(x+\epsilon z)dz_1\ldots dz_n.
  \]
  Since $\tilde K_\epsilon(x,x+\epsilon z)$ is actually independent of $x$, we have
  \[
    \frac{\partial h_\epsilon}{\partial x_j}(x)=\epsilon^n\int_{\mathbb{R}^n}\tilde K_\epsilon(x,x+\epsilon z)\frac{\partial h}{\partial x_j}(x+\epsilon z)dz_1\ldots dz_n
  \]
  for all $1\leq j\leq n$, or
  \[
    \frac{\partial h_\epsilon}{\partial x_j}(x)=\int_{\mathbb{R}^n}\tilde K_\epsilon(x,y)\frac{\partial h}{\partial y_j}(y)dy_1\ldots dy_n,
  \]
  so
  \begin{equation}
  \label{eq:h_derivative_convergence}
    \lim_{\epsilon\rightarrow 0^+}\frac{\partial h_\epsilon}{\partial x_j}(x)=\frac{\partial h}{\partial x_j}(x)\text{ for all }1\leq j\leq n.
  \end{equation}
  For $1\leq j,k\leq n$, we have
  \[
    \frac{\partial^2 h_\epsilon}{\partial x_j\partial x_k}(x)=\int_{\mathbb{R}^n}\frac{\partial}{\partial x_k}\tilde K_\epsilon(x,y)\frac{\partial h}{\partial y_j}(y)dy_1\ldots dy_n,
  \]
  but then \eqref{eq:tilde_K_derivative_identity} gives us
  \[
    \frac{\partial^2 h_\epsilon}{\partial x_j\partial x_k}(x)=\int_{\mathbb{R}^n}\frac{\partial}{\partial x_k}\tilde K_\epsilon(x,y)\left(\frac{\partial h}{\partial y_j}(y)-\frac{\partial h}{\partial y_j}(x)\right)dy_1\ldots dy_n,
  \]
  so continuity of $\frac{\partial h}{\partial y_j}$ and \eqref{eq:tilde_K_derivative_estimate} give us
  \begin{equation}
  \label{eq:h_hessian_convergence}
    \abs{\frac{\partial^2 h_\epsilon}{\partial x_j\partial x_k}(x)}\leq o\left(\frac{1}{\epsilon}\right)\text{ as }\epsilon\rightarrow 0^+.
  \end{equation}

  We define $K_\epsilon(x,y)=\tilde K_\epsilon(x,y)/h_\epsilon(x)$.  Clearly we have (1) and (2).  For $1\leq j\leq n$, \eqref{eq:tilde_K_derivative_relation} implies
  \[
    \frac{\partial}{\partial x_j}K_\epsilon(x,y)=-\frac{\partial}{\partial y_j}K_\epsilon(x,y)-K_\epsilon(x,y)\frac{1}{h_\epsilon(x)}\frac{\partial h_\epsilon}{\partial x_j}(x)
  \]
  for all $1\leq j\leq n$.  Since $\left(\frac{\partial}{\partial y_j}\right)^*=-\frac{\partial}{\partial y_j}-\frac{1}{h(y)}\frac{\partial h}{\partial y_j}(y)$, we have
  \begin{equation}
  \label{eq:X_K}
    \frac{\partial}{\partial x_j}K_\epsilon(x,y)=\left(\frac{\partial}{\partial y_j}\right)^*K_\epsilon(x,y)
    +E_j^h(x,y)K_\epsilon(x,y),
  \end{equation}
  where $E_j^h(x,y)=\frac{1}{h(y)}\frac{\partial h}{\partial y_j}(y)-\frac{1}{h_\epsilon(x)}\frac{\partial h_\epsilon}{\partial x_j}(x)$.  Note that \eqref{eq:h_derivative_convergence} gives us
  \begin{equation}
  \label{eq:E_h_estimate}
    \lim_{\epsilon\rightarrow 0^+}\sup_{\{x,y\in\mathbb{R}^n:\dist(x,y)\leq\epsilon\}}|E_j^h(x,y)|=0\text{ for all }1\leq j\leq n
  \end{equation}
  and \eqref{eq:h_hessian_convergence} gives us
  \begin{equation}
  \label{eq:E_h_derivative_estimate}
    \lim_{\epsilon\rightarrow 0^+}\sup_{\{x,y\in\mathbb{R}^n:\dist(x,y)\leq\epsilon\}}\epsilon\abs{\frac{\partial}{\partial x_k}E_j^h(x,y)}=0\text{ for all }1\leq j,k\leq n
  \end{equation}
  Since $\int_{\mathbb{R}^n}K_\epsilon(x,y)dV_y$ is constant, we have
  \[
    \int_{\mathbb{R}^n}\frac{\partial}{\partial x_j}K_\epsilon(x,y)dV_y=\frac{\partial}{\partial x_j}\left(\int_{\mathbb{R}^n}K_\epsilon(x,y)dV_y\right)=0.
  \]
  Stokes' Theorem immediately gives us
  \begin{equation}
  \label{eq:Stokes_consequence}
    \int_{\mathbb{R}^n}\left(\frac{\partial}{\partial y_j}\right)^*K_\epsilon(x,y)dV_y=0,
  \end{equation}
  so \eqref{eq:X_K} implies
  \begin{equation}
  \label{eq:vanishing_integral}
    \int_{\mathbb{R}^n}E_j^h(x,y)K_\epsilon(x,y)dV_y=0\text{ for all }1\leq j\leq n.
  \end{equation}

  Let $f\in C^1(M)$ be supported in our local coordinate chart, and set $f_\epsilon(x)=\int_{\mathbb{R}^n}f(y)K_\epsilon(x,y)dV_y$.  Using \eqref{eq:X_K}, we have
  \[
    \frac{\partial f_\epsilon}{\partial x_j}(x)=\int_{\mathbb{R}^n}f(y)\left(\frac{\partial}{\partial y_j}\right)^*K_\epsilon(x,y)dV_y
    +\int_{\mathbb{R}^n}f(y)E_j^h(x,y)K_\epsilon(x,y)dV_y.
  \]
  Integrating by parts in the first integral and using \eqref{eq:vanishing_integral} to rewrite the second integral, we have
  \[
    \frac{\partial f_\epsilon}{\partial x_j}(x)=\int_{\mathbb{R}^n}\frac{\partial f}{\partial y_j}(y)K_\epsilon(x,y)dV_y
    +\int_{\mathbb{R}^n}(f(y)-f(x))E_j^h(x,y)K_\epsilon(x,y)dV_y.
  \]
  From \eqref{eq:E_h_estimate} we obtain $\frac{\partial f_\epsilon}{\partial x_j}(x)\rightarrow \frac{\partial f}{\partial x_j}(x)$ as $\epsilon\rightarrow 0^+$ for all $1\leq j\leq n$.

  For $1\leq j,k\leq n$, we may differentiate again and use \eqref{eq:X_K} and \eqref{eq:vanishing_integral} to simplify and obtain
  \begin{multline*}
    \frac{\partial^2 f_\epsilon}{\partial x_j\partial x_k}(x)=\int_{\mathbb{R}^n}\frac{\partial f}{\partial y_j}(y)\left(\frac{\partial}{\partial y_k}\right)^*K_\epsilon(x,y)dV_y
    +\int_{\mathbb{R}^n}\frac{\partial f}{\partial y_j}(y)E_j^h(x,y)K_\epsilon(x,y)dV_y\\
    +\int_{\mathbb{R}^n}(f(y)-f(x))\left(\frac{\partial}{\partial x_k}E_j^h(x,y)K_\epsilon(x,y)+E_j^h(x,y)\frac{\partial}{\partial x_k}K_\epsilon(x,y)\right)dV_y.
  \end{multline*}
  Since $|f(y)-f(x)|\leq O(\epsilon)$ when $\dist(x,y)\leq\epsilon$, we may use \eqref{eq:tilde_K_derivative_estimate}, \eqref{eq:E_h_estimate}, and \eqref{eq:E_h_derivative_estimate} to show
  \[
    \abs{\frac{\partial^2 f_\epsilon}{\partial x_j\partial x_k}(x)-\int_{\mathbb{R}^n}\frac{\partial f}{\partial y_j}(y)\left(\frac{\partial}{\partial y_k}\right)^*K_\epsilon(x,y)dV_y}\rightarrow 0
  \]
  as $\epsilon\rightarrow 0^+$ for all $1\leq j,k\leq n$.

  Passing from the Hessian in local coordinates to the invariant Hessian with respect to arbitrary $C^1$ sections is a straightforward computation, and we omit the details.

\end{proof}

\bibliographystyle{amsplain}
\bibliography{harrington}
\end{document}